\documentclass[12 pt]{article}%
\usepackage{amsmath, amsfonts, amsthm, color,latexsym}
\usepackage{amsmath}
\usepackage{amsfonts}
\usepackage{amssymb}
\usepackage{color, soul}
\usepackage[all]{xy}
\usepackage{graphicx}%
\providecommand{\U}[1]{\protect\rule{.1in}{.1in}}
\allowdisplaybreaks[4]
\newtheorem{theorem}{Theorem}[section]
\newtheorem{proposition}[theorem]{Proposition}
\newtheorem{corollary}[theorem]{Corollary}
\newtheorem{example}[theorem]{Example}

\newtheorem{remark}[theorem]{Remark}

\newtheorem{lemma}[theorem]{Lemma}
\newtheorem{final remark}[theorem]{Final Remark}
\newtheorem{definition}[theorem]{Definition}
\allowdisplaybreaks[4]

\begin{document}

\title{Order continuity of Arens extensions of regular multilinear operators}
\author{Geraldo Botelho\thanks{Supported by CNPq Grant
304262/2018-8 and Fapemig Grant PPM-00450-17.}\,\, and  Luis Alberto Garcia\thanks{Supported by a CAPES scholarship.\newline 2020 Mathematics Subject Classification: 46A40, 46B42, 46G25, 47B65.\newline Keywords: Riesz spaces, Banach lattices, Arens extension, separate order continuity.
}}
\date{}
\maketitle

\begin{abstract}\noindent \,\,~First we give a counterexample showing that recent results on separate order continuity of Arens extensions of multilinear operators cannot be improved to get separate order continuity on the product of the whole of the biduals. Then we establish conditions on the operators and/or on the underlying Riesz spaces/Banach lattices so that the extensions are order continuous on the product of the whole biduals. We also prove that all Arens extensions of any regular multilinear operator are order continuous in at least one variable and we study when Arens extensions of regular homogeneous polynomials on a Banach lattice $E$ are order continuous on $E^{**}$.
\end{abstract}

\section{Introduction}

The second adjoint $u^{**}$ of a linear operator $u$, which is a bidual extension of $u$, is a powerful tool in several areas of mathematics. For multilinear operators, the same role has been played by Arens extensions, which have been extensively studied for the last 70 years since Arens' seminal paper \cite{arens}. In order to state the two recent results that have motivated our work, let us fix some notation. By $E^\sim$ we denote the order dual of a Riesz space $E$, hence $E^{\sim\sim} = (E^\sim)^\sim$ denotes its second order dual. For a Banach lattice $E$, $E^*$ denotes its topological dual, hence $E^{**}$ stands for its bidual. The symbols $(E^\sim)_n^{\sim}$ and $(E^*)_n^{*}$ stand for the corresponding subspaces formed by the order continuous functionals. The results that motivated our research are the following:

\medskip

\noindent $\bullet$ Buskes and Roberts (2019) \cite[Theorem 3.4]{Buskes}: If $A \colon E_1 \times \cdots \times E_m \longrightarrow F$ is an $m$-linear operator of order bounded variation between Riesz spaces, then its Arens extension $A^{[m+1]*} \colon E_1^{\sim\sim} \times \cdots \times E_m^{\sim\sim} \longrightarrow F^{\sim\sim}$ is separately order continuous on $(E_1^\sim)_n^{\sim} \times \cdots \times (E_m^\sim)_n^{\sim}.$

\medskip

\noindent
$\bullet$ Boyd, Ryan and Snigireva (2021) \cite[Theorem 1]{ryan1}: If $A \colon E_1 \times \cdots \times E_m \longrightarrow F$ is a regular $m$-linear operator between Banach lattices, with $F$ Dedekind complete, then its Arens extension $A^{[m+1]*} \colon E_1^{**} \times \cdots \times E_m^{**} \longrightarrow F^{**}$ is separately order continuous on $(E_1^*)_n^{*} \times \cdots \times (E_m^*)_n^{*}.$

The obvious question is whether or not these results can be improved to get order continuity on $E_1^{\sim\sim} \times \cdots \times E_m^{\sim\sim}$ and $E_1^{**} \times \cdots \times E_m^{**}$, respectively. By means of a counterexample we show that this is not the case (cf. Section 3). Actually our counterexample discloses an interesting phenomenon: for a certain regular bilinear form $A$ on $\ell_1 \times c_0$, its Arens extension $A^{***}$ is not separately order continuous on $\ell_1^{**} \times c_0^{**}$ (more  precisely, it is order continuous in the first variable but not in the second one), while the other Arens extension of $A$ is. Then we proceed to find conditions on the operator and/or on the underlying spaces so that all Arens extensions are separately order continuous on the product of the whole of the biduals. In Section 4 we prove that this holds for finite sums of multiplicative operators from Riesz spaces to Archimedean $f$-algebras, in particular for operators of finite type between arbitrary Riesz spaces. The main result of Section 4 (Theorem \ref{quasesoc}): (i) implies that all Arens extensions of any regular multilinear operator between Riesz spaces are order continuous in at least one variable, (ii) implies that all Arens extensions of a regular homogeneous polynomial from the Riesz space $E$ to a Riesz space $F$ are order continuous at the origin on $E^{\sim\sim}$, (iii) improves the results of Boyd, Ryan and Snigireva and of Buskes and Roberts for regular operators (see Remark \ref{remm}). In the final Section 5 we give sufficient conditions on the Banach lattices $E_1, \ldots, E_m$ so that Arens extensions of any regular $m$-linear operator from $E_1 \times \cdots \times E_m$ to an arbitrary Banach lattice $F$ are separately order continuous on $E_1^{**} \times \cdots \times E_m^{**}$; and conditions so that Arens extensions of regular homogeneous polynomials on a Banach lattice $E$ are order continuous on $E^{**}$.

In Section 2 we discuss briefly the notion of order continuity of linear operators and recall the characterization of the Arens extensions of regular multilinear operators between Riesz spaces that shall fit our purposes. Although these extensions are usually called Aron-Berner extensions in the case of operators between Banach spaces (see \cite{aronberner, livrosean}), for simplicity we shall refer to Arens extensions even for operators between Banach lattices.

\section{Background}

Our references to Riesz spaces, Banach lattices and regular linear operators are the canonical ones \cite{positiveoperators, nieberg, schaefer}.

The following three definitions of order convergence can be found in the literature (see \cite{abramovich}). A net $(x_{\alpha})_{\alpha\in\Omega}$  in a Riesz space $E$ is said to be:\\
$\bullet$ \textit{order convergent} to $x\in E$ if there is a net   $(y_{\alpha})_{\alpha\in\Omega}$  in $E$ such that $y_{\alpha}\downarrow 0$ and $|x_{\alpha}-x|\leq y_{\alpha}$ for every $\alpha\in\Omega$.\\
$\bullet$ \textit{1-convergent} to $x\in E$ if there are a net  $(y_{\alpha})_{\alpha\in\Omega}$  in $E$ and $\alpha_{0}\in \Omega$ such that $y_{\alpha}\downarrow 0$ and $|x_{\alpha}-x|\leq y_{\alpha}$ for every $\alpha\geq \alpha_{0}$.\\
$\bullet$ \textit{2-convergent} to $x\in E$ is there are a net  $(y_{\beta})_{\beta\in\Gamma}$  in $E$ such that $y_{\beta}\downarrow 0$ and for every $\beta\in \Gamma$ there exists $\alpha_{0}\in\Omega$ such that $|x_{\alpha}-x|\leq y_{\beta}$ for any $\alpha\geq \alpha_{0}$.

If $E$ is Dedekind complete, then the notions of 1-convergence and 2-convergence coincide \cite{abramovich}, but otherwise they may be different (see \cite[Example 1.4]{abramovich}).

Order continuity can be considered with respect to any of these three notions of order convergence: a linear operator $T \colon E \longrightarrow F$ between Riesz spaces is said to be {\it order continuous} ({\it 1-order continuous}, {\it 2-order continuous}, respectively) if  $(T(x_{\alpha}))_{\alpha\in\Omega}$ is order convergent (1-convergent, 2-convergent, respectively) to zero in $F$ whenever $(x_{\alpha})_{\alpha\in\Omega}$ is order convergent (1-convergent, 2-convergent, respectively) to zero in $E$.

The following coincidences are known (see \cite[Theorem 1.56]{positiveoperators} or \cite[Proposition 1.3.9]{nieberg}):\\
$\bullet$ If $F$ is Dedekind complete, then $T$ is order continuous if and only if $T$ is 1-order continuous.\\
$\bullet$ If $E$ and $F$ are Dedekind complete, then $T$ is order continuous if and only if $T$ is 1-order continuous if and only if $T$ is 2-order continuous.

In most cases in this paper we will investigate the order continuity of regular linear operators from $E^{\sim\sim}$ to $F^{\sim\sim}$, where $E$ and $F$ are Riesz spaces, or from $E^{**}$ to $F^{**}$, where $E$ and $F$ are Banach lattices. Since all these spaces are Dedekind complete, we are free to use any of the three notions of order continuity. We shall denote the order convergence by $x_\alpha \stackrel{o}{\longrightarrow} x$.

For the theory of regular multilinear operators and regular homogeneous polynomials we refer to \cite{bu, Buskes, loane}. An $m$-linear operator $A \colon E_1 \times \cdots \times E_m \longrightarrow F$ is {\it separately order continuous} ({\it separately 1-order continuous}, {\it separately 2-order continuous}) if for all $j \in \{1, \ldots, m\}$ and $x_k \in E_k, k = 1,\ldots, m, k\neq j$, the linear operator
$$x_j \in E_j \mapsto A(x_1, \ldots, x_m) \in F $$
is order continuous (1-order continuous, 2-order continuous). For the definition of joint order continuity see \cite[p.\,234]{ryan1}. If $A$ is a regular operator between Banach lattices with $F$ Dedekind complete, then $A$ is separately order continuous if and only if $A$ is jointly order continuous \cite[Theorem 2]{ryan1}.

Now we recall the description of the Arens extensions of regular multilinear operators between Riesz spaces as presented in \cite{lg}. By $J_E \colon E \longrightarrow E^{\sim\sim}$ we denote the canonical operator ($J_E(x)(x'') = x''(x)$), which happens to be a Riesz homomorphism.

Given Riesz spaces $E_1, \ldots, E_m,F$, the space of regular $m$-linear operators from $E_1 \times \cdots \times E_m$ to $F$ is denoted by ${\cal L}_r(E_1, \ldots, E_m;F)$. When $F$ is the scalar field we write ${\cal L}_r(E_1, \ldots, E_m)$. $S_m$ stands for the set of permutations of $\{1, \ldots, m\}$. For $\rho\in S_{m}$ and $k\in \{1,\ldots,m\}$, we fix the following notation:
$$E_{1},\ldots,\,_{\rho(1)}E,\ldots,\,_{\rho(k-1)}E,\ldots,E_{m}=\left\{ \begin{array}{cl}
E_{1},\ldots, E_{m} \mbox{~in this order} & \mbox{if}\,\ k=1, \\
E_{1},\ldots, E_{m} \mbox{~in this order, where}\\ E_{\rho(1)},\ldots,
   E_{\rho(k-1)} \mbox{~are removed} & \mbox{if}~ k=2,\ldots,m.
\end{array}\right.$$
For instance, $(E_1,\,_2E, E_3) = (E_1, E_3)$.
The same procedure defines the $(m-k+1)$-tuple $(x_{1},\ldots,\,_{\rho(1)}x,\ldots,\,_{\rho(k-1)}x,\ldots,x_{m})$ and the cartesian product $E_{1}\times\cdots \times\,_{\rho(1)}E \times\cdots \times\,_{\rho(k-1)}E\times\cdots \times E_{m}$.   Moreover, for $k=1,\ldots,m-1$, we write $$E_{1},\ldots,\,_{\rho(1)}E,\ldots,\,_{\rho(k)}E,\ldots,E_{m}=E_{1},\ldots, E_{m}$$ in this order, where $E_{\rho(1)},\ldots,E_{\rho(k)}$ are removed. In the same fashion we define the $(m-k)$-tuple $(x_{1},\ldots,_{\rho(1)}x,\ldots,_{\rho(k)}x,\ldots,x_{m})$ and the corresponding cartesian product.

  Finally, for $k=m$ we write $\mathcal{L}(E_{1},\ldots,\,_{\rho(1)}E,\ldots,\,_{\rho(k)}E,\ldots,E_{m};\mathbb{R})=\mathbb{R}.$

Let $k\in\{1,\ldots,m\}$, a permutation $\rho\in S_{m}$, Riesz spaces $E_{1},\ldots,E_{m}$ and an operator $A\in\mathcal{L}_{r}(E_{1},\ldots,\,_{\rho(1)}E,\ldots,\,_{\rho(k-1)}E,\ldots,E_{m})$ be given. For $x_{r}\in E_{r}, r\in \{1,\ldots,m\}\setminus \{\rho(1),\ldots,\rho(k)\}$, consider the linear functionals
$$A(x_{1},\ldots,\,_{\rho(1)}x,\ldots,\,_{\rho(k)}x;\bullet\,;\ldots, x_{m})\colon E_{\rho(k)}\longrightarrow \mathbb{R},$$
\begin{equation}\label{defpunto}
A(x_{1},\ldots,\,_{\rho(1)}x,\ldots,\,_{\rho(k)}x;\bullet\,;\ldots,x_{m})(x_{\rho(k)})
=A(x_{1},\ldots,\,_{\rho(1)}x,\ldots,\,_{\rho(k-1)}x,\ldots,x_{m}),
\end{equation}
where the dot $\bullet$ is placed at the $\rho(k)$-th coordinate.  Observe that for $k=m$ we have $A(x_{1},\ldots,\,_{\rho(1)}x,\ldots,\,_{\rho(m)}x;\bullet\,;\ldots,x_{m})=A\in E_{\rho(m)}^{\sim}$.

For every $x_{\rho(k)}^{\prime\prime} \in E_{\rho(k)}^{\sim\sim}$, the map
\begin{equation}\label{defope}\overline{x_{\rho(k)}^{\prime\prime}}^{\rho}\colon \mathcal{L}_{r}(E_{1},\ldots,\,_{\rho(1)}E,\ldots,\,_{\rho(k-1)}E,\ldots,E_{m})\longrightarrow \mathcal{L}_{r}(E_{1},\ldots,\,_{\rho(1)}E,\ldots,\,_{\rho(k)}E,\ldots,E_{m}),
\end{equation}
$$\overline{x_{\rho(k)}^{\prime\prime}}^{\rho}(A)(x_{1},\ldots,\,_{\rho(1)}x,
\ldots,\,_{\rho(k)}x,\ldots,x_{m})=x_{\rho(k)}^{\prime\prime}(A(x_{1},\ldots,\,_{\rho(1)}x,\ldots,\,_{\rho(k)}x;\bullet\,;\ldots,x_{m})),$$ 
is a regular linear operator and $\Big|\overline{x_{\rho(k)}^{\prime\prime}}^{\rho} \Big|\leq \overline{|x_{\rho(k)}^{\prime\prime}|}^{\rho}$. Furthermore, if $0\leq x_{\rho(k)}^{\prime\prime} \in E_{\rho(k)}^{\sim\sim}$ then the operator $\overline{x_{\rho(k)}^{\prime\prime}}^{\rho}$ is positive \cite[Proposition 2.1]{lg}.

Given a permutation $\rho\in S_{m}$ and a regular $m$-linear operator $A \colon E_{1}\times\cdots\times E_{m} \longrightarrow F$, the Arens extension of $A$ with respect to $\rho$ is the operator $AR_{m}^{\rho}(A)\colon E_{1}^{\sim\sim}\times
\cdots\times E_{m}^{\sim\sim}\longrightarrow F^{\sim\sim}$ defined by $$ AR_{m}^{\rho}(A)(x_{1}^{\prime\prime},\ldots,x_{m}^{\prime\prime})(y^{\prime})=\big(\overline{x_{
\rho(m)}^{\prime\prime}}^{\rho}\circ\cdots\circ\overline
{x_{\rho(1)}^{\prime\prime}}^{\rho}\big)(y^{\prime}\circ A)$$
for every $y^{\prime}\in F^{\sim}$. According to \cite[Theorem 2.2]{lg},  $AR_{m}^{\rho}(A)$ is a regular $m$-linear operator that extends $A$ in the sense that $AR_{m}^{\rho}(A)\circ (J_{E_{1}},\ldots,J_{E_{m}})=J_{F}\circ A.$ Moreover,  $AR_{m}^{\rho}(A)$ is positive for positive $A$.

The extension $A^{[m+1]*}$ from \cite{ryan1, Buskes} is recovered by considering the permutation $\theta(m) = 1,  \theta(m-1) = 2, \ldots,\theta(2) = m-1, \theta(1) = m$, that is, $AR_m^\theta(A) = A^{*[m+1]}$. In particular, $AR_2^\theta(A) = A^{***}$ in the bilinear case $m = 2$.

\section{The counterexample}
Consider the positive bilinear form
$$A\colon \ell_{1}\times c_{0}\longrightarrow \mathbb{R}~, A((x_{n})_{n=1}^{\infty},(y_{n})_{n=1}^{\infty})=\displaystyle\sum_{n=1}^{\infty}x_{n}y_{n}.$$
The two Arens extensions of $A$ shall be denoted by $A^{***} =AR_2^\theta(A)$ and $AR_2^{\rm id}(A)$, where id is the identity permutation. As announced, we shall prove that $AR_2^{\rm id}(A)$ is separately order continuous on $\ell_1^{**} \times c_0^{**}$ and that $A^{***}$ is order continuous in the first variable but not in the second one. Although everything can be proved directly to this bilinear form, to avoid unnecessary repetitions we shall apply some results that will be proved later.

From Theorem \ref{quasesoc} we know that  $A^{***} \colon \ell_{1}^{\ast\ast}\times c_{0}^{\ast\ast}\longrightarrow\mathbb{R}$ is order continuous in the first variable. 
Suppose that $A^{***}$ is order continuous in the second variable, that is, 
for every $x^{\ast\ast}\in \ell_{1}^{\ast\ast}$, the linear functional  $A^{***}(x^{\ast\ast},\bullet)\colon c_{0}^{\ast\ast}\longrightarrow \mathbb{R}$ is order continuous. Denoting by 
$\psi\colon \ell_{1}\longrightarrow c_{0}^{*}$ 
the canonical isometric isomorphism, note that $\psi$ and $\psi^{-1}\colon c_{0}^{\ast}\longrightarrow \ell_{1}$, $ \psi^{-1}(\varphi)=(\varphi(e_{n}))_{n=1}^{\infty}$, are positive operators, hence  $\psi$ is a Riesz homomorphism 
\cite[Theorem 2.15]{positiveoperators}. 
Moreover, $\psi(x)=A(x,\bullet)$ for every $x\in \ell_{1}$.

\medskip

\noindent \textbf{Claim 1.} $\psi^{\ast\ast}(x^{\ast\ast})=A^{***}(x^{\ast\ast},\bullet)$ for every $x^{\ast\ast}\in \ell_{1}^{\ast\ast}$.

Indeed, given $y^{\ast\ast}\in c_{0}^{\ast\ast}$ and $x\in \ell_{1}$, bearing in mind that $A^{***} =AR_2^\theta(A)$,
$$\psi^{\ast}(y^{\ast\ast})(x)=y^{\ast\ast}(\psi(x))=y^{\ast\ast}(A(x,\bullet))=\overline{y^{\ast\ast}}^{\theta}(A)(x),$$
so $\psi^{\ast}(y^{\ast\ast})=\overline{y^{\ast\ast}}^{\theta}(A)$. Therefore, for  $x^{\ast\ast}\in \ell_{1}^{\ast\ast}$ and $y^{\ast\ast}\in c_{0}^{\ast\ast}$,
$$\psi^{\ast\ast}(x^{\ast\ast})(y^{\ast\ast})=x^{\ast\ast}(\psi^{\ast}(y^{\ast\ast}))=x^{\ast\ast}(\overline{y^{\ast\ast}}^{\theta}(A))=AB_{2}^{\theta}(A)(x^{\ast\ast},y^{\ast\ast}) = A^{***}(x^{\ast\ast}, \bullet)(y^{\ast\ast}).$$

\medskip

\noindent\textbf{Claim 2.} $x^{**}\in \ell_{1}^{\ast\ast}$ is order continuous on $\ell_1^*$ if and only if $\psi^{\ast\ast}(x^{\ast\ast})\in c_{0}^{\ast\ast\ast}$ is order continuous on $c_0^{**}$.

Let $x^{**}\in \ell_{1}^{\ast\ast}$ be such that $\psi^{\ast\ast}(x^{\ast\ast})\in c_{0}^{\ast\ast\ast}$ is order continuous on $c_0^{**}$. Supposing that $x^{\ast\ast}$ fails to be order continuous on $\ell_1^*$, the positive functional  $|x^{\ast\ast}|\in \ell_{1}^{\ast\ast}$ is not order continuous on $\ell_1^*$ either \cite[Theorem 1.56]{positiveoperators}. Then there is a net  $(x_{\alpha}^{\ast})_{\alpha\in\Omega}$ in $\ell_{1}^{\ast}$ such that $x_{\alpha}^{\ast}\downarrow 0$ but $\displaystyle\inf_{\alpha\in \Omega}|x^{\ast\ast}|(x_{\alpha}^{\ast})> 0$. 
For each $\alpha\in\Omega$ let $y_{\alpha}^{\ast\ast}\in c_{0}^{\ast\ast}$ be such that $\psi^{\ast}(y_{\alpha}^{\ast\ast})=x_{\alpha}^{\ast}$.  Thus $$y_{\alpha}^{\ast\ast}=(\psi^{\ast})^{-1}(x_{\alpha}^{\ast})=(\psi^{-1})^{\ast}(x_{\alpha}^{\ast})\downarrow 0 \text{ in } c_{0}^{\ast\ast}$$
because $(\psi^{-1})^{\ast}$ is positive and order continuous \cite[Theorem 1.73]{positiveoperators}. 
By assumption $\psi^{\ast\ast}(x^{\ast\ast})$ is order  continuous  on $c_0^{**}$, so is $|\psi^{\ast\ast}(x^{\ast\ast})|$ \cite[Theorem 1.56]{positiveoperators}. Since $\psi$ is a Riesz homomorphism, $\psi^{\ast\ast}$ is as well, so $\psi^{\ast\ast}(|x^{\ast\ast}|)(y_{\alpha}^{\ast\ast})=|\psi^{\ast\ast}(x^{\ast\ast})|(y_{\alpha}^{\ast\ast})\downarrow 0$, from which it follows that
$$0=\inf_{\alpha\in\Omega}\psi^{\ast\ast}(|x^{\ast\ast}|)(y_{\alpha}^{\ast\ast})=
\inf_{\alpha\in\Omega}|x^{\ast\ast}|(\psi^{\ast}(y_{\alpha}^{\ast\ast}))
=\inf_{\alpha\in\Omega}|x^{\ast\ast}|(x_{\alpha}^{\ast})>0.$$
This contradiction proves that $x^{\ast\ast}\in\ell_{1}^{\ast\ast}$ is order continuous on $\ell_1^*$ . The reverse implication if straightforward.

\medskip
\noindent\textbf{Claim 3.} $\ell_{1}^{\ast\ast}$ contains a functional that fails to be order continuous on $\ell_1^*$.

Let $c$ be the space of convergent real sequences and consider the positive linear functional $\varphi \in c^* $ given by  $\varphi((x_{n})_{n=1}^{\infty})=\displaystyle\lim_{n\rightarrow\infty}x_{n}$. Since $c$ is a majorizing subspace of $\ell_{\infty}$, $\varphi$ admits a positive extension $\widetilde{\varphi}\in \ell_{\infty}^{\ast}$ \cite[Theorem 1.32]{positiveoperators}. Suppose that $\widetilde{\varphi}$ is $\sigma$-order continous on $\ell_\infty$. For each $n\in\mathbb{N}$ let $x_{n}=(1,\ldots,1,0,\ldots) = e_1 + \cdots + e_n\in c$ and $y=(1,1,\ldots)\in c$. Note that $0\leq x_{n}\uparrow y$ and, since  $\widetilde{\varphi}$ is a positive $\sigma$-order continuous operator,   $0\leq \widetilde{\varphi}(x_{n})\uparrow \widetilde{\varphi}(y)$ (see \cite[p.\,46]{positiveoperators}). So, 
$$1=\widetilde{\varphi}(y)=\sup_{n\in\mathbb{N}}\widetilde{\varphi}(x_{n})=0,$$
which proves that $\widetilde{\varphi}\in \ell_{\infty}^{\ast}$ is not $\sigma$-order continuous on $\ell_\infty$. Considering the canonical Riesz isomorphism $\phi\colon\ell_{\infty}\longrightarrow \ell_{1}^{\ast}$, 
there is $z^{\ast\ast}\in \ell_{1}^{\ast\ast}$ such that $\phi^{\ast}(z^{\ast\ast})=\widetilde{\varphi}$. Since $\widetilde{\varphi}\in \ell_{\infty}^{\ast}$  fails to be $\sigma$-order continuous, there is a sequence  
$(z_{n})_{n=1}^{\infty}$ in $\ell_{\infty}$ such that $z_{n}\downarrow 0$ and $\displaystyle\inf_{n\in\mathbb{N}}\widetilde{\varphi}(z_{n})>0$. Furthermore, there are $y_{n}^{\ast}\in \ell_{1}^{\ast}$, $n \in \mathbb{N}$, such that  $\phi^{-1}(y_{n}^{\ast})=z_{n}$, hence  $y_{n}^{\ast}=\phi(z_{n})\downarrow $ because $\phi$ is positive, so $0\leq y_{n}^{\ast}\downarrow$. Suppose that there exists $y^{\ast}\in \ell_{1}^{\ast}$ such that  $0<y^{\ast}\leq y_{n}^{\ast}$ for every $n\in\mathbb{N}$. On the one hand, as $\phi^{-1}$ is positive,  $$0 \leq \phi^{-1}(y^{\ast})\leq \phi^{-1}(y_{n}^{\ast})=z_{n} \mbox{ for every } n,$$
from which we conclude that $\phi^{-1}(y^{\ast})=0$, and so $y^{\ast}=0$ once $\phi^{-1}$ is injective. This shows that $y_{n}^{\ast}\downarrow 0$ in $\ell_{1}^{\ast}$. On the other hand,
\begin{align*}
\inf_{n\in\mathbb{N}}z^{\ast\ast}(y_{n}^{\ast})=
\inf_{n\in\mathbb{N}}(\phi^{\ast})^{-1}(\widetilde{\varphi})(y_{n}^{\ast})
=\inf_{n\in\mathbb{N}}(\phi^{-1})^{\ast}(\widetilde{\varphi})(y_{n}^{\ast})
=\inf_{n\in\mathbb{N}}\widetilde{\varphi}(\phi^{-1}(y_{n}^{\ast}))
=\inf_{n\in\mathbb{N}}\widetilde{\varphi}(z_{n})>0,
\end{align*}
proving that $z^{\ast\ast}$ is not order continuous on $\ell_\infty$, as claimed. 

Finally, combining Claims 1 and 2 we have that $A^{***}(z^{**}, \bullet)= \psi^{**}(z^{**})$  is not order continuous on $c_0^{**}$. We have established that $A^{***}$ is  order continuous in the first variable and fails to be order continuous in the second variable. 

As to the other Arens extension of $A$, namely $AR_{2}^{\rm id}(A)\colon \ell_{1}^{\ast\ast}\times c_{0}^{\ast\ast}\longrightarrow\mathbb{R}$, 
since $c_{0}^{\ast} = \ell_{1}$ has order continuous norm, Corollary \ref{cor1} guarantees that $AR_{2}^{id}(A)$ is separately order continuous, hence jointly order continuous by  \cite[Theorem 2]{ryan1}.

Since the bilinear form $A$ is regular and of bounded order variation, this example shows that the results of Buskes and Roberts and of Boyd, Ryan and Snigireva quoted in the introduction cannot be improved to get separate order continuity on the product of the whole biduals.

\section{Operators between Riesz spaces}

In this section we present our results on order continuity on the whole of the biduals of Arens extensions of multilinear operators on Riesz spaces. The main result of the section, namely Theorem \ref{quasesoc}, is a multipurpose result: in this section it will be used to prove that Arens extensions of regular homogeneous polynomials are always order continuous at the origin on the whole of the bidual of the domain space, to extend \cite[Theorem 1]{ryan1}, to show that Arens extensions are always order continuous in at least one variable and, finally, it will be helpful a couple of times in the next section.

Recall that a {\it Riesz algebra} $\mathcal{A}$ é is a Riesz space which is an associative algebra with respect to a produtc $\ast$ such that $x\ast y\geq 0$ for all $x,y\in \mathcal{A}^{+}$. And that a Riesz algebra $(\mathcal{A},\ast)$ is an {\it $f$-algebra} if $x\wedge y=0$ in $\mathcal{A}$ implies that $(x\ast z)\wedge y=(z\ast x)\wedge y=0$ for every $z\in \mathcal{A}^{+}.$ %
If $(\mathcal{A},\ast)$ is an $f$-algebra, then the Arens product $\odot$ , defined as follows, makes ${\cal A}^{\sim\sim}$ an $f$-algebra \cite{yilma}: 
for $ x\in \mathcal{A}$, $y^{\prime}\in \mathcal{A}^{\sim}$ and $x^{\prime\prime}, y^{\prime\prime}\in \mathcal{A}^{\sim\sim}$,
\begin{align*}
y^{\prime}\cdot x&\colon \mathcal{A}\longrightarrow \mathbb{R}~,~ (y^{\prime}\cdot x)(y)=y^{\prime}(x\ast y).\\
 x^{\prime\prime}\diamond y^{\prime}&\colon \mathcal{A}\longrightarrow \mathbb{R}~,~(x^{\prime\prime}\diamond y^{\prime})(y)=x^{\prime\prime}(y^{\prime}\cdot y).\\
x^{\prime\prime}\odot y^{\prime\prime}&\colon \mathcal{A}^{\sim}\longrightarrow \mathbb{R}~,~(x^{\prime\prime}\odot y^{\prime\prime})(z^{\prime})=x^{\prime\prime}(y^{\prime\prime}\diamond z^{\prime} ).
\end{align*}

 An operator $A\in \mathcal{L}_{r}(E_{1},\ldots,E_{m};\mathcal{A})$ is {\it multiplicative} if there are regular linear operators $T_{i}\colon E_{i}\longrightarrow\mathcal{A}, i=1,\ldots,m$, such that $A(x_{1},\ldots,x_{m})=T_{1}(x_{1})\ast\cdots\ast T_{m}(x_{m})$ for all $x_{1}\in E_{1},\ldots,x_{m}\in E_{m}$.

Since Arens extensions $AR_{m}^{\rho}(A)$ of multilinear operators $A$ are mappings between Dedekind complete spaces, we can use any of the three notions of order continuous linear operators to investigate the separate order continuity of $AR_{m}^{\rho}(A)$.

\begin{proposition}\label{propo3}
Let $E_{1},\ldots,E_{m}$ be Riesz spaces and $ (\mathcal{A},\ast)$ be an Archimedean $f$-algebra. If $A\in\mathcal{L}_{r}(E_{1},\ldots,E_{m};\mathcal{A})$ is a finite sum of multiplicative operators, then all Arens extensions of $A$,  $AR_{m}^{\rho}(A)$, $\rho\in S_{m}$, coincide and are separately order continuous.
\end{proposition}

\begin{proof}
Given a multiplicative operator  $B\in\mathcal{L}_{r}(E_{1},\ldots,E_{m};\mathcal{A})$, let $T_{i}\colon E_{i}\longrightarrow\mathcal{A}, i=1,\ldots,m$, be such that  $B(x_{1},\ldots,x_{m})=T_{1}(x_{1})\ast\cdots\ast T_{m}(x_{m})$ for all $x_{1}\in E_{1},\ldots,x_{m}\in E_{m}$. By \cite[Remark 3.3 and the proof of Theorem 3.2]{lg} we have that, for each $\rho\in S_{m}$ and all $x_{1}^{\prime\prime}\in E_{1}^{\sim\sim},\ldots,x_{m}^{\prime\prime}\in E_{m}^{\sim\sim}$,
$$AR_{m}^{\rho}(B)(x_{1}^{\prime\prime},\ldots,x_{m}^{\prime\prime})=T_{\rho(m)}^{\prime\prime}(x_{\rho(m)}^{\prime\prime})\odot \cdots\odot T_{\rho(1)}^{\prime\prime}(x_{\rho(1)}^{\prime\prime}).$$
The Arens product $\odot$ makes $\mathcal{A}^{\sim\sim}$ a Dedekind complete, hence Archimedean, commutative $f$-algebra \cite[Corollaries 3.5 and 3.6]{yilma}, so 
$$AR_{m}^{\rho}(B)(x_{1}^{\prime\prime},\ldots,x_{m}^{\prime\prime})=T_{1}^{\prime\prime}(x_{1}^{\prime\prime})\odot \cdots\odot T_{m}^{\prime\prime}(x_{m}^{\prime\prime}),$$
which gives, in particular, that all Arens extensions of $A$ coincide. 
In order to check that $AR_{m}^{\rho}(B)$ is separately order continuous, let  $j\in\{1,\ldots.m\}$, $x_{i}^{\prime\prime}\in  E_{i}^{\sim\sim}, i=1,\ldots,m$, with $i\neq j$ be given and let $(x_{\alpha_{j}}^{\prime\prime})_{\alpha_{j}\in\Omega_{j}}$ be a net in $E_{j}^{\sim\sim}$ such that $x_{\alpha_{j}}^{\prime\prime} \xrightarrow{\,\, o \,\,} 0$. There exists a net $(z_{\alpha_{j}}^{\prime\prime})_{\alpha_{j}\in\Omega_{j}}$  and $\alpha_{j_{0}}\in \Omega_{j}$ such that $z_{\alpha_{j}}^{\prime\prime} \downarrow 0$ and $|x_{\alpha_{j}}^{\prime\prime}|\leq z_{\alpha_{j}}^{\prime\prime}$ for every  $\alpha_{j}\geq \alpha_{j_{0}}$. The functional $$\varphi:=|T_{1}^{\prime\prime}(x_{1}^{\prime\prime})|\odot\cdots\odot |T_{j-1}^{\prime\prime}(x_{j-1}^{\prime\prime})|\odot |T_{j+1}^{\prime\prime}(x_{j+1}^{\prime\prime})|\odot \cdots\odot |T_{m}^{\prime\prime}(x_{m}^{\prime\prime})|\in \mathcal{A}^{\sim\sim}$$
 is positive. Using again that the product $\odot$ is commutative and \cite[Exercise 12, p.\,131]{positiveoperators},
\begin{align*}
|AR_{m}^{\rho}&(B)(x_{1}^{\prime\prime},\ldots,x_{\alpha_{j}}^{\prime\prime},\ldots,x_{m}^{\prime\prime})|=| T_{1}^{\prime\prime}(x_{1}^{\prime\prime})\odot\cdots\odot  T_{j}^{\prime\prime}(x_{\alpha_{j}}^{\prime\prime})\odot \cdots\odot T_{m}^{\prime\prime}(x_{m}^{\prime\prime})|\\
&=|T_{1}^{\prime\prime}(x_{1}^{\prime\prime})|\odot\cdots\odot  |T_{j-1}^{\prime}(x_{j-1}^{\prime\prime})|\odot |T_{j}^{\prime\prime}(x_{\alpha_{j}}^{\prime\prime})|\odot |T_{j+1}^{\prime\prime}(x_{j+1}^{\prime\prime})|\odot \cdots\odot |T_{m}^{\prime\prime}(x_{m}^{\prime\prime})|\\
&=|T_{j}^{\prime\prime}(x_{\alpha_{j}}^{\prime\prime})|\odot \big(|T_{1}^{\prime\prime}(x_{1}^{\prime\prime})|\odot\cdots\odot |T_{j-1}^{\prime\prime}(x_{j-1}^{\prime\prime})|\odot |T_{j+1}^{\prime\prime}(x_{j+1}^{\prime\prime})|\odot \cdots\odot |T_{m}^{\prime\prime}(x_{m}^{\prime\prime})|\big)\\
&=|T_{j}^{\prime\prime}(x_{\alpha_{j}}^{\prime\prime})|\odot \varphi\leq |T_{j}^{\prime\prime}|(|x_{\alpha_{j}}^{\prime\prime}|)\odot \varphi\leq |T_{j}^{\prime\prime}|(z_{\alpha_{j}}^{\prime\prime})\odot \varphi.
\end{align*}
Now it is enough to prove that $|T_{j}^{\prime\prime}|(z_{\alpha_{j}}^{\prime\prime})\odot \varphi\downarrow 0$. Let $0\leq y^{\prime}\in \mathcal{A}^{\sim}$ be given. Then 
$\varphi\diamond y^{\prime}$ is positive and, since $|T_{j}^{\prime\prime}|$ is order continuous and positive   \cite[Theorems 1.56 and 1.73]{positiveoperators}, 
$$\big(|T_{j}^{\prime\prime}|(z_{\alpha_{j}}^{\prime\prime})\odot \varphi\big)(y^{\prime})=|T_{j}^{\prime\prime}|(z_{\alpha_{j}}^{\prime\prime})(\varphi\diamond y^{\prime})\downarrow 0,$$
from which it follows that $\big(|T_{j}^{\prime\prime}|(z_{\alpha_{j}}^{\prime\prime})\odot \varphi\big)\downarrow 0$ \cite[Teorema 1.18]{positiveoperators} and gives the separate order continuity of $AR_{m}^{\rho}(B)$.

The linearity of the correspondence $A \mapsto AR_{m}^{\rho}(A)$ gives the result for finite sums of multiplicative operators. 
\end{proof}

Since scalar-valued Riesz multimorphisms are multiplicative {\cite[Theorem 6]{kus}, Proposition  \ref{propo3}  yields the following.

\begin{corollary} All Arens extensions of a finite sum of scalar-valued Riesz multimorphisms coincide and are separately order continuous.
\end{corollary}

An operator $A \in {\cal L}_r(E_1, \ldots, E_m;F)$ is of {\it finite type} if there are $n\in\mathbb{N}$, functionals $\varphi_{j}^{i}\in E_{i}^{\sim}$ and vectors $y_{j}\in F$, $j=1,\ldots,n$, $i=1,\ldots,m$, such that
$$A(x_{1},\ldots,x_{m})=\sum_{j=1}^{n} \varphi_{j}^{1}(x_{1})\cdots  \varphi_{j}^{m}(x_{m})y_{j} \text{ for all } x_{i}\in E_{i}, i=1,\ldots,m.$$

\begin{corollary}  All Arens extensions of a multilinear operator of finite type coincide, are of finite type and are separately order continuous.
\end{corollary}

\begin{proof} It is not difficult to check that if $A \in {\cal L}_r(E_1, \ldots, E_m)$ is separately order continuous and $y \in F$, then the operator
$$(x_1, \ldots, x_m) \in E_1 \times \cdots \times E_m \mapsto A(x_1, \ldots, x_m)y \in F, $$
is separately order continuous as well. Now the result follows from Proposition \ref{propo3} and from its proof.
\end{proof}

To proceed to the main results of the section we need some preparation.

\begin{lemma} \label{le2}
Let $E_{1},\ldots,E_{m}, F$ be Riesz spaces with $F$ Dedekind complete and $(B_{\alpha})_\alpha$ be a net in $\mathcal{L}_{r}(E_{1},\ldots,E_{m}; F)$. Then $B_{\alpha}\downarrow 0$ if and only if $B_{\alpha}(x_{1},\ldots,x_{m})\downarrow 0$ in $F$ for all $x_{1}\in E_{1}^{+},\ldots, x_{m}\in E_{m}^{+}$.
\end{lemma}

\begin{proof} It is straightforward that $B_{\alpha}\downarrow 0$ if $B_{\alpha}(x_{1},\ldots,x_{m})\downarrow 0$ in $F$ for all $x_{1}\in E_{1}^{+},\ldots, x_{m}\in E_{m}^{+}$. %
We prove the other implication by induction on $m$. The case $m=1$ follows from the  Riesz-Kantorovich Theorem \cite[Theorem 1.18]{positiveoperators}. Assume that the result holds for $n$ 
and let $(B_{\alpha})_\alpha$ be a net in $\mathcal{L}_{r}(E_{1},\ldots,E_{n+1}; F)$ such that $B_{\alpha}\downarrow 0$. 
Consider the canonical Riesz isomorphism
$$\psi\colon \mathcal{L}_{r}(E_{1},\ldots,E_{n+1}; F)\longrightarrow \mathcal{L}_{r}(E_{1};\mathcal{L}_{r}(E_{2}\ldots,E_{n+1}; F)).$$
We have $0\leq \psi(B_{\alpha})\downarrow $ because $\psi$ is positive. Let $T \in \mathcal{L}_{r}(E_{1};\mathcal{L}_{r}(E_{2}\ldots,E_{n+1}; F))$ be such that  $0\leq T\leq \psi(B_{\alpha})$ for every $\alpha$. Since $\psi^{-1}$ is positive, $0 \leq \psi^{-1}(T)\leq B_{\alpha}\downarrow 0$, hence $0\leq \psi^{-1}(T)\leq 0$, which proves that 
$\psi(B_{\alpha})\downarrow 0$ in $\mathcal{L}_{r}(E_{1};\mathcal{L}_{r}(E_{2}\ldots,E_{n+1}; F))$. The linear case of the result gives that $\psi(B_{\alpha})(x_{1})\downarrow 0$ in $\mathcal{L}_{r}(E_{2}\ldots,E_{n+1}; F)$ for every $x_{1}\in E_{1}^{+}$. The induction hypothesis gives that, regardless of the $x_{2}\in E_{2}^{+},\ldots,x_{n+1}\in E_{n+1}^{+}$,
$$B_{\alpha}(x_{1},x_{2},\ldots,x_{n+1})=\psi(B_{\alpha})(x_{1})(x_{2},\ldots,x_{n+1})\downarrow 0,$$
completing the proof.
\end{proof}


\begin{lemma}\label{le1}
  Let $E_{1},\ldots,E_{m}$ be Riesz spaces, $\rho\in S_{m}$, $k \in \{1,\ldots,m\}$ and $x_{\rho(k)}^{\prime\prime}\in (E_{\rho(k)}^{\sim})_{n}^{\sim}$. Then the operator
$$\overline{x_{\rho(k)}^{\prime\prime}}^{\rho}\colon \mathcal{L}_{r}(E_{1},\ldots,_{\rho(1)}E,\ldots,_{\rho(k-1)}E,\ldots,E_{m})\longrightarrow \mathcal{L}_{r}(E_{1},\ldots,_{\rho(1)}E,\ldots,_{\rho(k)}E,\ldots,E_{m}),$$
defined in \mbox{\rm(\ref{defope})}, is order continuous.
\end{lemma}
\begin{proof}
We already know that $\overline{x_{\rho(k)}^{\prime\prime}}^{\rho}$ is
a regular linear operator and $\big|\overline{x_{\rho(k)}^{\prime\prime}}^{\rho} \big|\leq \overline{|x_{\rho(k)}^{\prime\prime}|}^{\rho}$. Let $(A_{\alpha})_{\alpha\in\Omega}$ be a net in  $\mathcal{L}_{r}(E_{1},\ldots,_{\rho(1)}E,\ldots,_{\rho(k-1)}E,\ldots,E_{m})$ such that $A_{\alpha}\xrightarrow{\,\, o\,\, } 0$. Then there are  a net $(B_{\alpha})_{\alpha\in\Omega}$ in $\mathcal{L}_{r}(E_{1},\ldots,_{\rho(1)}E,\ldots,_{\rho(k-1)}E,\ldots,E_{m})$ and $\alpha_{0}\in\Omega$ such that $B_{\alpha}\downarrow 0$ and $|A_{\alpha}|\leq B_{\alpha}$ for every $\alpha\geq \alpha_{0}$. Thus, $$\big|\overline{x_{\rho(k)}^{\prime\prime}}^{\rho}(A_{\alpha})\big|\leq \big|\overline{x_{\rho(k)}^{\prime\prime}}^{\rho}\big|(|A_{\alpha}|)\leq \overline{|x_{\rho(k)}^{\prime\prime}|}^{\rho}(|A_{\alpha}|)\leq \overline{|x_{\rho(k)}^{\prime\prime}|}^{\rho}(B_{\alpha}) \text{ for every } \alpha\geq \alpha_{0}.$$
For $x_{i}\in E_{i}^{+}$, $i\in\{1,\ldots,m\}\setminus\{\rho(1),\ldots,\rho(k)\}$, Lemma \ref{le2} gives
$$B_{\alpha}(x_{1},\ldots,_{\rho(1)}x,\ldots,_{\rho(k)}x;\bullet;\ldots,x_{m})\downarrow 0.$$
Since $x_{\rho(k)}^{\prime\prime}$ is order continuous,  $|x_{\rho(k)}^{\prime\prime}|$ is a positive order continuous operator  \cite[Theorem 1.56]{positiveoperators}, so $|x_{\rho(k)}^{\prime\prime}|(B_{\alpha}(x_{1},\ldots,_{\rho(1)}x,\ldots,_{\rho(k)}x;\bullet;\ldots,x_{m}))\downarrow 0$, that is,
$$\overline{|x_{\rho(k)}^{\prime\prime}|}^{\rho}(B_{\alpha})(x_{1},\ldots,_{\rho(1)}x,\ldots,_{\rho(k)}x,\ldots,x_{m})\downarrow 0.$$
Calling on Lemma \ref{le2} once again it follows that $ \overline{|x_{\rho(k)}^{\prime\prime}|}^{\rho}(B_{\alpha})\downarrow 0$, proving that $\overline{x_{\rho(k)}^{\prime\prime}}^{\rho}$ is order continuous.
\end{proof}

\begin{theorem}\label{quasesoc}
 Let $E_{1},\ldots,E_{m}, F$ be Riesz spaces, $\rho\in S_{m}$ and $A\in\mathcal{L}_{r}(E_{1},\ldots,E_{m};F)$.
\begin{enumerate}
\item[\rm (a)] For all $j\in \{1,\ldots,m\}$, $x_{\rho(i)}^{\prime\prime}\in E_{\rho(i)}^{\sim\sim}, i=1,\ldots,j-1$, and  $x_{\rho(i)}^{\prime\prime}\in (E_{\rho(i)}^{\sim})_{n}^{\sim}, i=j+1,\ldots,m$, the operator
    \begin{equation}\label{eqiou}x''_{\rho(j)} \in E_{\rho(j)}^{\sim\sim} \mapsto AR_{m}^{\rho}(A)(x_{1}^{\prime\prime},\ldots,x_{{\rho(j)}}^{\prime\prime},\ldots,x_{m}^{\prime\prime}) \in F^{\sim\sim}
    \end{equation}
    is order continuous on $E_{\rho(j)}^{\sim\sim}$.
\item[\rm (b)]  $AR_{m}^{\rho}(A)$ is separately order continuous on $(E_{1}^{\sim})_{n}^{\sim}\times\cdots\times (E_{m}^{\sim})_{n}^{\sim}$.
\item[\rm (c)] $AR_{m}^{\rho}(A)$ is order continuous in the $\rho(m)$-th variable on the whole of $E_{\rho(m)}^{\sim\sim}$.
\end{enumerate}
\end{theorem}
\begin{proof} It is plain that (b) and (c) follow from (a) (for (c) just take $j = m$ in (a)). To prove (a),
 take $j\in \{1,\ldots,m\}$, $x_{\rho(i)}^{\prime\prime}\in E_{\rho(i)}^{\sim\sim}, i=1,\ldots,j-1$,  and $x_{\rho(i)}^{\prime\prime}\in (E_{\rho(i)}^{\sim})_{n}^{\sim}, i=j+1,\ldots,m$. Given a net  $(x_{\alpha_{\rho(j)}}^{\prime\prime})_{\alpha_{\rho(j)}\in\Omega_{\rho(j)}}$ in $E_{\rho(j)}^{\sim\sim}$ such that $x_{\alpha_{\rho(j)}}^{\prime\prime} \xrightarrow{\,\, o \,\, } 0$, there are a net $(z_{\alpha_{\rho(j)}}^{\prime\prime})_{\alpha_{\rho(j)}\in\Omega_{\rho(j)}}$ in $E_{\rho(j)}^{\sim\sim}$ and $\alpha_{\rho(j)_{0}}$ such that $z_{\alpha_{\rho(j)}}^{\prime\prime}\downarrow 0$ and $|x_{\alpha_{\rho(j)}}^{\prime\prime}|\leq z_{\alpha_{\rho(j)}}^{\prime\prime}$ for every $\alpha_{\rho(j)}\geq\alpha_{\rho(j)_{0}}$. Let  $A_{1}, A_{2}\in\mathcal{L}_{r}(E_{1},\ldots,E_{m};F)$ be positive operators such that $A=A_{1}-A_{2}$ and put $B:=A_{1}+A_{2}$. Of course $B$ is positive. Denoting the operator in (\ref{eqiou}) by $AR_{m}^{\rho}(A)_{x_{\rho(1)}^{\prime\prime},\ldots,x_{\rho(j-1)}^{\prime\prime},x_{\rho(j+1)}^{\prime\prime},\ldots,x_{\rho(m)}^{\prime\prime}}$, for every $\alpha_{\rho(j)}\geq\alpha_{\rho(j)_{0}}$,
\begin{align*}
|AR_{m}^{\rho}(A)_{x_{\rho(1)}^{\prime\prime},\ldots,x_{\rho(j-1)}^{\prime\prime},x_{\rho(j+1)}^{\prime\prime},\ldots,x_{\rho(m)}^{\prime\prime}}&(x_{\alpha_{\rho(j)}}^{\prime\prime})|
=|AR_{m}^{\rho}(A)(x_{1}^{\prime\prime},\ldots,x_{\alpha_{\rho(j)}}^{\prime\prime},\ldots,x_{m}^{\prime\prime})|\\
&\leq |AR_{m}^{\rho}(A)|(|x_{1}^{\prime\prime}|,\ldots,|x_{\alpha_{\rho(j)}}^{\prime\prime}|,\ldots,|x_{m}^{\prime\prime}|)\\
&= |AR_{m}^{\rho}(A_{1}-A_{2})|(|x_{1}^{\prime\prime}|,\ldots,|x_{\alpha_{\rho(j)}}^{\prime\prime}|,\ldots,|x_{m}^{\prime\prime}|)\\
&= |AR_{m}^{\rho}(A_{1})-AR_{m}^{\rho}(A_{2})|(|x_{1}^{\prime\prime}|,\ldots,|x_{\alpha_{\rho(j)}}^{\prime\prime}|,\ldots,|x_{m}^{\prime\prime}|)\\
&\leq \big(AR_{m}^{\rho}(A_{1})+AR_{m}^{\rho}(A_{2}) \big) (|x_{1}^{\prime\prime}|,\ldots,|x_{\alpha_{\rho(j)}}^{\prime\prime}|,\ldots,|x_{m}^{\prime\prime}|)\\
&= AR_{m}^{\rho}(B)(|x_{1}^{\prime\prime}|,\ldots,|x_{\alpha_{\rho(j)}}^{\prime\prime}|,\ldots,|x_{m}^{\prime\prime}|)\\
&=AR_{m}^{\rho}(B)_{|x_{\rho(1)}^{\prime\prime}|,\ldots,|x_{\rho(j-1)}^{\prime\prime}|,|x_{\rho(j+1)}^{\prime\prime}|,\ldots,|x_{\rho(m)}^{\prime\prime}|}(|x_{\alpha_{\rho(j)}}^{\prime\prime}|)\\
&\leq AR_{m}^{\rho}(B)_{|x_{\rho(1)}^{\prime\prime}|,\ldots,|x_{\rho(j-1)}^{\prime\prime}|,|x_{\rho(j+1)}^{\prime\prime}|,\ldots,|x_{\rho(m)}^{\prime\prime}|}(z_{\alpha_{\rho(j)}}^{\prime\prime}).
\end{align*}
As Arens extensions of positive operators are positive, it holds
$$0\leq AR_{m}^{\rho}(B)_{|x_{\rho(1)}^{\prime\prime}|,\ldots,|x_{\rho(j-1)}^{\prime\prime}|,|x_{\rho(j+1)}^{\prime\prime}|,\ldots,|x_{\rho(m)}^{\prime\prime}|}(z_{\alpha_{\rho(j)}}^{\prime\prime})\downarrow.$$ Calling $T:=\overline{|x_{\rho(m)}^{\prime\prime}|}^{\rho}\circ\cdots\circ \overline{|x_{\rho(j+1)}^{\prime\prime}|}^{\rho}$, since each $|x_{\rho(i)}^{\prime\prime}|, i=j+1,\ldots,m$, is order continuous, by Lemma \ref{le1} it follows that  $\overline{|x_{\rho(i)}^{\prime\prime}|}^{\rho}$ is order continuous, so $T$ is order continuous and positive. On the other hand, it is plain that, for every positive $y^{\prime}\in F^{\sim}$,
$$S:=\big(\overline{|x_{\rho(j-1)}^{\prime\prime}|}^{\rho}\circ \cdots\circ \overline{|x_{\rho(1)}^{\prime\prime}|}^{\rho}\big)(y^{\prime}\circ B)\in \mathcal{L}_{r}(E_{1},\ldots,_{\rho(1)}E,\ldots,_{\rho(j-1)}E,\ldots,E_{m})$$
is positive. From $z_{\alpha_{\rho(j)}}^{\prime\prime}\downarrow 0$ we conclude that $\overline{z_{\alpha_{\rho(j)}}^{\prime\prime}}^{\rho}(S)\downarrow 0$, therefore  $T(\overline{z_{\alpha_{\rho(j)}}^{\prime\prime}}^{\rho}(S))\downarrow 0$. In this fashion, for every positive $y^{\prime}\in F^{\sim}$,
\begin{align*}
AR_{m}^{\rho}(B)&_{|x_{\rho(1)}^{\prime\prime}|,\ldots,|x_{\rho(j-1)}^{\prime\prime}|,|x_{\rho(j+1)}^{\prime\prime}|,\ldots,|x_{\rho(m)}^{\prime\prime}|}(z_{\alpha_{\rho(j)}}^{\prime\prime})(y^{\prime})=AR_{m}^{\rho}(B)(|x_{1}^{\prime\prime}|,\ldots,z_{\alpha_{\rho(j)}}^{\prime\prime},\ldots,|x_{m}^{\prime\prime}|)(y^{\prime})\\
&=\big(\overline{|x_{\rho(m)}^{\prime\prime}|}^{\rho}\circ\cdots\circ \overline{|x_{\rho(j+1)}^{\prime\prime}|}^{\rho}\circ \overline{z_{\alpha_{\rho(j)}}^{\prime\prime}}^{\rho}\circ \overline{|x_{\rho(j-1)}^{\prime\prime}|}^{\rho}\circ \cdots\circ \overline{|x_{\rho(1)}^{\prime\prime}|}^{\rho}\big)(y^{\prime}\circ B)\\
&=\big(T\circ \overline{z_{\alpha_{\rho(j)}}^{\prime\prime}}^{\rho}\circ \overline{|x_{\rho(j-1)}^{\prime\prime}|}^{\rho}\circ \cdots\circ \overline{|x_{\rho(1)}^{\prime\prime}|}^{\rho}\big)(y^{\prime}\circ B)\\
&=T\big(\big(\overline{z_{\alpha_{\rho(j)}}^{\prime\prime}}^{\rho}\circ \overline{|x_{\rho(j-1)}^{\prime\prime}|}^{\rho}\circ \cdots\circ \overline{|x_{\rho(1)}^{\prime\prime}|}^{\rho}\big)(y^{\prime}\circ B)\big)\\
&=T\big(\overline{z_{\alpha_{\rho(j)}}^{\prime\prime}}^{\rho}\big(\big(\overline{|x_{\rho(j-1)}^{\prime\prime}|}^{\rho}\circ \cdots\circ \overline{|x_{\rho(1)}^{\prime\prime}|}^{\rho}\big)(y^{\prime}\circ B)\big)\big)=T(\overline{z_{\alpha_{\rho(j)}}^{\prime\prime}}^{\rho}(S))\downarrow 0.
\end{align*}
Lemma \ref{le2} gives that  $AR_{m}^{\rho}(B)_{|x_{\rho(1)}^{\prime\prime}|,\ldots,|x_{\rho(j-1)}^{\prime\prime}|,|x_{\rho(j+1)}^{\prime\prime}|,\ldots,|x_{\rho(m)}^{\prime\prime}|}(z_{\alpha_{\rho(j)}}^{\prime\prime})\downarrow 0$, and this allows us to conclude that  $AR_{m}^{\rho}(A)_{x_{\rho(1)}^{\prime\prime},\ldots,x_{\rho(j-1)}^{\prime\prime},x_{\rho(j+1)}^{\prime\prime},\ldots,x_{\rho(m)}^{\prime\prime}}$ is order continuous.
\end{proof}



\begin{remark}\label{remm} \rm Theorem \ref{quasesoc} improves \cite[Theorem 1]{ryan1} in the sense that it  holds for all Arens extensions, it holds for operators between Riesz spaces, it drops the assumption of $F$ being Dedekind complete and it assures the order continuity on the whole bidual in one of the variables. And, for regular operators, it improves \cite[Theorem 3.4]{Buskes} by taking into account all Arens extensions and by assuring the order continuity on the whole bidual in one of the variables. In particular,  Theorem \ref{quasesoc}(b) provides an alternative proof of \cite[Theorem 1]{ryan1} and of \cite[Theorem 3.4]{Buskes} for regular operators between Riesz spaces and (c) shows that $A^{\ast[m+1]}=AR_{m}^{\theta}(A)$ is order continuous in the first variable on the whole of $E_1^{**}$.
\end{remark}

Recall that an $m$-homogeneous polynomial $P \colon E \longrightarrow F$ between Riesz spaces is positive if the corresponding symmetric $m$-linear operator $\check P$ is positive. And $P$ is regular, in symbols $P \in {\cal P}_r(^mE;F)$, if $P$ can be written as the difference of two positive polynomials.

The Arens extensions of a regular polynomial $P \in {\cal P}_r(^mE;F)$ are the polynomials associated to the Arens extensions of $\check P$, that is: for $\rho \in S_m$, the Arens extension of $P$ with respect to $\rho$ is the polynomial
$$AR_m^\rho(P) \colon E^{\sim\sim} \longrightarrow F^{\sim\sim}~,~ AR_m^\rho(P)(x'') =AR_m^\rho(\check P)(x'', \ldots, x''). $$
In \cite[Theorem 3.5]{Buskes} it is proved that $AR_m^\theta(P)$ is order continuous on $(E^\sim)_n^{\sim}$.  We can go a bit further at the origin:

\begin{proposition}\label{respol} All Arens extensions of a polynomial $P \in {\cal P}_r(^mE;F)$ are order continuous at the origin on $E^{\sim\sim}$, meaning that $AR_m^\rho(P)(x_{\alpha}^{\prime\prime})\xrightarrow{\,\, o\,\,} 0$ in $F^{\sim\sim}$ for every $\rho \in S_m$ and any $(x_{\alpha}^{\prime\prime})_{\alpha\in\Omega}$ in $E^{\sim\sim}$ such that $x_{\alpha}^{\prime\prime}\xrightarrow{\,\, o \,\,} 0$ in $E^{\sim\sim}$.
\end{proposition}

\begin{proof} Write $P=P_{1}-P_{2}$, where $P_1$ and $P_2$ are positive $m$-homogeneous polynomials, and let 
$\check{P}_{1}, \check{P}_{2}\colon E^{m}\longrightarrow F$ be the positive symmetric $m$-linear operators associated to $P_1$ and $P_2$, respectively. 
Let $(x_{\alpha}^{\prime\prime})_{\alpha\in\Omega}$ be a net in $E^{\sim\sim}$ such that $x_{\alpha}^{\prime\prime}\xrightarrow{\,\, o \,\,} 0$. There are a net $(z_{\alpha}^{\prime\prime})_{\alpha\in\Omega}$ in $E^{\sim\sim}$ and $\alpha_{0}\in\Omega$ such that $z_{\alpha}^{\prime\prime}\downarrow 0$ and $|x_{\alpha}^{\prime\prime}|\leq z_{\alpha}^{\prime\prime}$ for every $\alpha\geq \alpha_{0}$. For a permutation $\rho \in S_m$, we know from Theorem \ref{quasesoc} that the operator
$$x'' \in E^{\sim\sim} \mapsto AR_m^{\rho}(\check{P}_1 + \check{P}_2)(z''_{\alpha_0}, \ldots, z''_{\alpha_0}, x'', z''_{\alpha_0}, \ldots, z''_{\alpha_0}), $$
where $x''$ is placed at the $\rho(m)$-th coordinate, is order continuous. For  $\alpha\geq \alpha_{0}$ we have $z_{\alpha}^{\prime\prime}\leq z_{\alpha_{0}}^{\prime\prime}$, so, using that $AR_m^{\rho}(\check{P}_1 + \check{P}_2)$ is positive,
\begin{align*}
|AR_m^\rho(P)(x_{\alpha}^{\prime\prime})|&=|AR_m^\rho(P_1 - P_2)(x_{\alpha}^{\prime\prime})| = |AR_m^\rho((P_1 - P_2)^{\vee})(x_{\alpha}^{\prime\prime}, \ldots , x_{\alpha}^{\prime\prime})| \\& =|AR_{m}^{\rho}(\check{P}_{1}-\check{P}_{2})(x_{\alpha}^{\prime\prime},\ldots,x_{\alpha}^{\prime\prime})|\leq |AR_{m}^{\rho}(\check{P}_{1}-\check{P}_{2})|(|x_{\alpha}^{\prime\prime}|,\ldots,|x_{\alpha}^{\prime\prime}|)\\
&\leq |AR_{m}^{\rho}(\check{P}_{1}-\check{P}_{2})|(z_{\alpha}^{\prime\prime},\ldots,z_{\alpha}^{\prime\prime}) \leq AR_{m}^{\rho}(\check{P}_{1}+\check{P}_{2})(z_{\alpha}^{\prime\prime},\ldots,z_{\alpha}^{\prime\prime}) \\
& \leq AR_{m}^{\rho}(\check{P}_{1}+\check{P}_{2})(z_{\alpha_{0}}^{\prime\prime},\ldots,z_{\alpha_{0}}^{\prime\prime},
z_{\alpha}^{\prime\prime},z_{\alpha_{0}}^{\prime\prime},\ldots,z_{\alpha_{0}}^{\prime\prime})\downarrow 0.
\end{align*}
This proves that  $AR_m^\rho(P)(x_{\alpha}^{\prime\prime})\xrightarrow{\,\, o\,\,} 0$.
\end{proof}

In \cite{nakano} it is proved that, for a regular homogeneous polynomial, order continuity at one point does not imply order continuity at every point in general. Anyway, the result above shall be useful later.

\section{Operators between Banach lattices}
In this section we give conditions on the Banach lattices $E_1, \ldots, E_m$ so that, for every Banach lattice $F$, all Arens extensions of any regular $m$-linear operator from $E_1 \times \cdots \times E_m$ to $F$ are separately order continuous on $E_1^{**} \times \cdots \times E_m^{**}$. Consequences on order continuity of extensions of regular homogeneous polynomials shall also be obtained.

 If the dual $E^{\ast}$ of a Banach lattice $E$ has order continuous norm, then    $E^{\ast\ast}=(E^{\ast})_{n}^{\ast}$ \cite[Theorem 2.4.2]{nieberg}. So, the following is immediate from Theorem \ref{quasesoc} .

\begin{corollary}\label{cor1}
Let $E_{1},\ldots, E_{m}, F$ be Banach lattices, $A\in\mathcal{L}_{r}(E_{1},\ldots,E_{m};F)$ and $\rho\in S_{m}$. If $E_{j}^{\ast}$ has order continuous norm  for $j\in\{1,\ldots,m\}$, $j\neq \rho(1)$, then the Arens extension $AR_{m}^{\rho}(A)$ of $A$ is separately order continuous on $E_1^{**} \times \cdots \times E_m^{**}$.
\end{corollary}

The next result makes clear what type of condition should be asked to get order continuity of Arens extensions on the product of the whole of the biduals.

\begin{proposition} Let $m \geq 2$ and $E_{1}, \ldots,E_{m}$ be Banach lattices such that the Arens extension $A^{\ast[m+1]}$ of any form $A \in {\cal L}_r(E_1, \ldots, E_m)$ is separately order continuous  on $E_{1}^{\ast\ast}\times\cdots\times  E_{m}^{\ast\ast}$. Then, for every operator  $T\in\mathcal{L}_{r}(E_{i};E_{j}^{\ast}),\ i,j=1,\ldots,m$, $i\neq j$, the functional $T^{\ast\ast}(x_{i}^{\ast\ast})$ is order continuous on $E_{j}^{\ast\ast}$  for every $x_{i}^{\ast\ast}\in E_{i}^{\ast\ast}$.
\end{proposition}

\begin{proof} Let $i,j=1,\ldots,m, i\neq j$, and $T\in \mathcal{L}_{r}(E_{i};E_{j}^{\ast})$ be given. For $k=1,\ldots,m, i \neq k\neq j$, choose  $0 \neq \varphi_{k}\in E_{k}^{\ast}$ and consider  the regular $m$-linear form $$A\colon E_{1}\times\cdots\times E_{m}\longrightarrow \mathbb{R}~,~A(x_{1},\ldots,x_{m})=\Bigg(\displaystyle\prod_{ \substack {k=1 \\ k\neq i, j}}^{m}\varphi_{k}(x_{k})\Bigg)T(x_{i})(x_{j}).$$
Of course we can assume $i<j$. Using the Davie--Gamelin description of the Arens extensions \cite{gamelin}, for $x_l^{**} \in E_l^{**}$ and nets $(x_{\alpha_{l}})_{\alpha_{l}\in \Omega_{l}}$ in $E_{l}$ such that $x_{l}^{\ast\ast}=\omega^{\ast}-\displaystyle\lim_{\alpha_{l}}J_{E_{l}}(x_{\alpha_{l}}), l=1,\ldots,m$, we have
\begin{align*}
A^{\ast[m+1]}&(x_{1}^{\ast\ast},\ldots,x_{i}^{\ast\ast},\ldots,x_{j}^{\ast\ast},\ldots,x_{m}^{\ast\ast})=\lim_{\alpha_{1}}\cdots\lim_{\alpha_{i}}\cdots\lim_{\alpha_{j}}\cdots\lim_{\alpha_{m}} A(x_{\alpha_{1}},\ldots,x_{\alpha_{m}})\\
&=\lim_{\alpha_{1}}\cdots\lim_{\alpha_{i}}\cdots\lim_{\alpha_{j}}\cdots\lim_{\alpha_{m}} \Bigg(\displaystyle\prod_{ \substack {k=1 \\ k\neq i, j}}^{m}\varphi_{k}(x_{\alpha_{k}})\Bigg)T(x_{\alpha_{i}})(x_{\alpha_{j}})\\
&=\lim_{\alpha_{1}}\cdots\lim_{\alpha_{i}}\cdots\lim_{\alpha_{j}}\cdots\lim_{\alpha_{m-1}} \Bigg(\displaystyle\prod_{ \substack {k=1 \\ k\neq i, j}}^{m-1}\varphi_{k}(x_{\alpha_{k}})\Bigg)T(x_{\alpha_{i}})(x_{\alpha_{j}})\lim_{\alpha_{m}} J_{E_{m}}(x_{\alpha_{m}})(\varphi_{m})\\
&=\lim_{\alpha_{1}}\cdots\lim_{\alpha_{i}}\cdots\lim_{\alpha_{j}}\cdots\lim_{\alpha_{m-1}} \Bigg(\displaystyle\prod_{ \substack {k=1 \\ k\neq i, j}}^{m-1}\varphi_{k}(x_{\alpha_{k}})\Bigg)T(x_{\alpha_{i}})(x_{\alpha_{j}}) x_{m}^{\ast\ast}(\varphi_{m})\\
&\,\,\,  \vdots\\
&=x_{m}^{\ast\ast}(\varphi_{m})\cdots x_{j+1}^{\ast\ast}(\varphi_{j+1})\lim_{\alpha_{1}}\cdots\lim_{\alpha_{i}}\cdots\lim_{\alpha_{j}}\Bigg(\displaystyle\prod_{ \substack {k=1 \\ k\neq i}}^{j-1}\varphi_{k}(x_{\alpha_{k}})\Bigg)T(x_{\alpha_{i}})(x_{\alpha_{j}}) \\
&=\prod_{k=j+1}^{m}x_{k}^{\ast\ast}(\varphi_{k})\lim_{\alpha_{1}}\cdots\lim_{\alpha_{i}}\cdots\lim_{\alpha_{j-1}}\Bigg(\displaystyle\prod_{ \substack {k=1 \\ k\neq i}}^{j-1}\varphi_{k}(x_{\alpha_{k}})\Bigg)\lim_{\alpha_{j}} T(x_{\alpha_{i}})(x_{\alpha_{j}}) \\
&=\prod_{k=j+1}^{m}x_{k}^{\ast\ast}(\varphi_{k})\lim_{\alpha_{1}}\cdots\lim_{\alpha_{i}}\cdots\lim_{\alpha_{j-1}}\Bigg(\displaystyle\prod_{ \substack {k=1 \\ k\neq i}}^{j-1}\varphi_{k}(x_{\alpha_{k}})\Bigg) x_{j}^{\ast\ast}( T(x_{\alpha_{i}}))\\
&\,\,\,  \vdots\\
&=\prod_{ \substack{k=i+1\\ k\neq j}}^{m}x_{k}^{\ast\ast}(\varphi_{k})\lim_{\alpha_{1}}\cdots\lim_{\alpha_{i}}\Bigg(\displaystyle\prod_{ \substack {k=1}}^{i-1}\varphi_{k}(x_{\alpha_{k}})\Bigg) x_{j}^{\ast\ast}( T(x_{\alpha_{i}}))\\
&=\prod_{ \substack{k=i+1\\ k\neq j}}^{m}x_{k}^{\ast\ast}(\varphi_{k})\lim_{\alpha_{1}}\cdots\lim_{\alpha_{i-1}}\Bigg(\displaystyle\prod_{ \substack {k=1}}^{i-1}\varphi_{k}(x_{\alpha_{k}})\Bigg) \lim_{\alpha_{i}}T^{\ast}(x_{j}^{\ast\ast})( x_{\alpha_{i}})\\
&=\prod_{ \substack{k=1\\ k\neq i, j}}^{m}x_{k}^{\ast\ast}(\varphi_{k}) \lim_{\alpha_{i}} J_{E_{i}}( x_{\alpha_{i}})(T^{\ast}(x_{j}^{\ast\ast}))=\Bigg(\prod_{ \substack{k=1\\ k\neq i, j}}^{m}x_{k}^{\ast\ast}(\varphi_{k})\Bigg)x_{i}^{\ast\ast}(T^{\ast}(x_{j}^{\ast\ast}))\\
&=\Bigg(\prod_{ \substack{k=1\\ k\neq i, j}}^{m}x_{k}^{\ast\ast}(\varphi_{k})\Bigg) T^{\ast\ast}(x_{i}^{\ast\ast})(x_{j}^{\ast\ast}).
\end{align*}
Choosing $x_{k}\in E_{k}$ so that $\varphi(x_{k})=1$, $i \neq k \neq j$, we get
\begin{align*}
A^{\ast[m+1]}(J_{E_{1}}(x_{1}),\ldots,x_{i}^{\ast\ast},\ldots, x_{j}^{\ast\ast},\ldots,J_{E_{m}}(x_{m})) 
=T^{\ast\ast}(x_{i}^{\ast\ast})(x_{j}^{\ast\ast}).
\end{align*}
Since $A^{\ast[m+1]}$ is separately order continuous by assumption, the functional $T^{\ast\ast}(x_{i}^{\ast\ast})$ is order continuous for every $x_{i}^{\ast\ast}\in E_{i}^{\ast\ast}$.
\end{proof}

 Although the next results hold, with the obvious modifications, for all Arens extensions $AR_{m}^{\rho}(A)$ of a regular $m$-linear operator $A$, to make the proofs more readable we shall restrict ourselves to the extension $A^{\ast[m+1]}=AR_{m}^{\theta}(A)$.

\begin{lemma}\label{lema} Let
 $E_{1}, \ldots,E_{m}$ be Banach lattices,  $A\in\mathcal{L}_{r}(E_{1},\ldots,E_{m})$ and $i\in\{1,\ldots,m\}$. If $x_{j}\in E_{j}, j=1,\ldots,i-1$, and $x_{j}^{\ast\ast}\in E_{j}^{\ast\ast}, j=i+1,\ldots,m$, then the operator  $$A^{\ast[m+1]}(J_{E_{1}}(x_{1}),\ldots,J_{E_{i-1}}(x_{i-1}),\bullet,x_{i+1}^{\ast\ast},\ldots,x_{m}^{\ast\ast})\colon E_{i}^{\ast\ast}\longrightarrow\mathbb{R}$$ is $\omega^{\ast}$-continuous and
$$A^{\ast[m+1]}(J_{E_{1}}(x_{1}),\ldots,J_{E_{i-1}}(x_{i-1}),x_{i}^{\ast\ast},\ldots,x_{m}^{\ast\ast})=\big(\overline{x_{i}^{\ast\ast}}^{\theta}\circ \cdots\circ \overline{x_{m}^{\ast\ast}}^{\theta}\big)(A)(x_{1},\ldots,x_{i-1}).$$
\end{lemma}

\begin{proof}
 Let $(x_{\alpha_{i}}^{\ast\ast})_{\alpha_{i}\in \Omega_{i}}$ be a net in  $E_{i}^{\ast\ast}$ such that $x_{\alpha_{i}}^{\ast\ast}\xrightarrow{\,\, \omega^{\ast} \,\,} x_{i}^{\ast\ast}\in E_{i}^{\ast\ast}$. For every $x_{i}^{\ast}\in E_{i}^{\ast}$ we have $x_{i}^{\ast\ast}(x_{i}^{\ast})=\displaystyle\lim_{\alpha_{i}}x_{\alpha_{i}}^{\ast\ast}(x_{i}^{\ast})$. Given $x_{j}\in E_{j}, j=1,\ldots,i-1$ and $x_{j}^{\ast\ast}\in E_{j}^{\ast\ast}, j=i+1,\ldots,m$,
\begin{align*}
A^{\ast[m+1]}(J_{E_{1}}(x_{1}),&\ldots,J_{E_{i-1}}(x_{i-1}),x_{i}^{\ast\ast},\ldots,x_{m}^{\ast\ast})\\
&=\big(\overline{J_{E_{1}}(x_{1})}^{\theta}\circ\cdots\circ \overline{J_{E_{i-1}}(x_{i-1})}^{\theta}\circ \overline{x_{i}^{\ast\ast}}^{\theta}\circ \cdots\circ \overline{x_{m}^{\ast\ast}}^{\theta}\big)(A)\\
&=\overline{J_{E_{1}}(x_{1})}^{\theta}\big(\big(\overline{J_{E_{2}}(x_{1})}^{\theta}\circ\cdots\circ \overline{J_{E_{i-1}}(x_{i-1})}^{\theta}\circ \overline{x_{i}^{\ast\ast}}^{\theta}\circ \cdots\circ \overline{x_{m}^{\ast\ast}}^{\theta}\big)(A)\big)\\
&=J_{E_{1}}(x_{1})\big(\big(\overline{J_{E_{2}}(x_{2})}^{\theta}\circ\cdots\circ \overline{J_{E_{i-1}}(x_{i-1})}^{\theta}\circ \overline{x_{i}^{\ast\ast}}^{\theta}\circ \cdots\circ \overline{x_{m}^{\ast\ast}}^{\theta}\big)(A)\big)\\
&=\big(\overline{J_{E_{2}}(x_{2})}^{\theta}\circ\cdots\circ \overline{J_{E_{i-1}}(x_{i-1})}^{\theta}\circ \overline{x_{i}^{\ast\ast}}^{\theta}\circ \cdots\circ \overline{x_{m}^{\ast\ast}}^{\theta}\big)(A)(x_{1})\\
&=\overline{J_{E_{2}}(x_{2})}^{\theta}\big(\big(\overline{J_{E_{3}}(x_{3})}^{\theta}\circ\cdots\circ \overline{J_{E_{i-1}}(x_{i-1})}^{\theta}\circ \overline{x_{i}^{\ast\ast}}^{\theta}\circ \cdots\circ \overline{x_{m}^{\ast\ast}}^{\theta}\big)(A)\big)(x_{1})\\
&=J_{E_{2}}(x_{2})\big(\big(\big(\overline{J_{E_{3}}(x_{3})}^{\theta}\circ\cdots\circ \overline{J_{E_{i-1}}(x_{i-1})}^{\theta}\circ \overline{x_{i}^{\ast\ast}}^{\theta}\circ \cdots\circ \overline{x_{m}^{\ast\ast}}^{\theta}\big)(A)\big)(x_{1},\bullet)\big)\\
&=\big(\big(\overline{J_{E_{3}}(x_{3})}^{\theta}\circ\cdots\circ \overline{J_{E_{i-1}}(x_{i-1})}^{\theta}\circ \overline{x_{i}^{\ast\ast}}^{\theta}\circ \cdots\circ \overline{x_{m}^{\ast\ast}}^{\theta}\big)(A)\big)(x_{1},\bullet)(x_{2})\\
&=\big(\overline{J_{E_{3}}(x_{3})}^{\theta}\circ\cdots\circ \overline{J_{E_{i-1}}(x_{i-1})}^{\theta}\circ \overline{x_{i}^{\ast\ast}}^{\theta}\circ \cdots\circ \overline{x_{m}^{\ast\ast}}^{\theta}\big)(A)(x_{1},x_{2})\\
&~\,\,\, \vdots\\
&=\big(\overline{J_{E_{i-1}}(x_{i-1})}^{\theta}\circ \overline{x_{i}^{\ast\ast}}^{\theta}\circ \cdots\circ \overline{x_{m}^{\ast\ast}}^{\theta}\big)(A)(x_{1},\ldots,x_{i-2})\\
&=\overline{J_{E_{i-1}}(x_{i-1})}^{\theta}\big(\big(\overline{x_{i}^{\ast\ast}}^{\theta}\circ \cdots\circ \overline{x_{m}^{\ast\ast}}^{\theta}\big)(A)\big)(x_{1},\ldots,x_{i-2})\\
&=J_{E_{i-1}}(x_{i-1})\big(\big(\big(\overline{x_{i}^{\ast\ast}}^{\theta}\circ \cdots\circ \overline{x_{m}^{\ast\ast}}^{\theta}\big)(A)\big)(x_{1},\ldots,x_{i-2},\bullet)\big)\\
&=\big(\overline{x_{i}^{\ast\ast}}^{\theta}\circ \cdots\circ \overline{x_{m}^{\ast\ast}}^{\theta}\big)(A)(x_{1},\ldots,x_{i-2},\bullet)(x_{i-1})\\
&=\big(\overline{x_{i}^{\ast\ast}}^{\theta}\circ \cdots\circ \overline{x_{m}^{\ast\ast}}^{\theta}\big)(A)(x_{1},\ldots,x_{i-2},x_{i-1})\\
&=\overline{x_{i}^{\ast\ast}}^{\theta}\big(\big(\overline{x_{i-1}^{\ast\ast}}^{\theta}\circ \cdots\circ \overline{x_{m}^{\ast\ast}}^{\theta}\big)(A)\big)(x_{1},\ldots,x_{i-2},x_{i-1})\\
&=x_{i}^{\ast\ast}\big(\big(\big(\overline{x_{i-1}^{\ast\ast}}^{\theta}\circ \cdots\circ \overline{x_{m}^{\ast\ast}}^{\theta}\big)(A)\big)(x_{1},\ldots,x_{i-2},x_{i-1},\bullet)\big)\\
&\stackrel{(\Delta)}{=}\displaystyle\lim_{\alpha_{i}}x_{\alpha_{i}}^{\ast\ast}\big(\big(\big(\overline{x_{i-1}^{\ast\ast}}^{\theta}\circ \cdots\circ \overline{x_{m}^{\ast\ast}}^{\theta}\big)(A)\big)(x_{1},\ldots,x_{i-2},x_{i-1},\bullet))\\
&=\displaystyle\lim_{\alpha_{i}}A^{\ast[m+1]}(J_{E_{1}}(x_{1}),\ldots,J_{E_{i-1}}(x_{i-1}),x_{\alpha_{i}}^{\ast\ast},\ldots,x_{m}^{\ast\ast}),
\end{align*}
where, in $(\Delta)$, we used that $\big(\big(\overline{x_{i-1}^{\ast\ast}}^{\theta}\circ \cdots\circ \overline{x_{m}^{\ast\ast}}^{\theta}\big)(A)\big)(x_{1},\ldots,x_{i-2},x_{i-1},\bullet)\in E_{i}^{\ast}$.
\end{proof}

\begin{definition}\rm Let $\cal P$ be a property of linear functionals on Banach lattices. We say that:\\
$\bullet$ A form $A\colon E_{1}^{\ast\ast}\times\cdots\times  E_{m}^{\ast\ast} \longrightarrow \mathbb{R}$, where $E_1, \ldots, E_m$ are Banach lattices, {\it has $\mathcal{P}$-separately} if for all $j \in \{1, \ldots, m\}$ and $x_{i}^{\ast\ast}\in E_{i}^{\ast\ast}, i=1,\ldots,m$, $i\neq j$, the functional
$$A_{x_{1}^{\ast\ast},\ldots, x_{j-1}^{\ast\ast},x_{j+1}^{\ast\ast},\ldots,x_{m}^{\ast\ast}}\colon E_{j}^{\ast\ast}\longrightarrow \mathbb{R}~,~x_j^{**} \mapsto A(x_1^{**}, \ldots, x_m^{**}),$$ has property $\mathcal{P}$.\\
$\bullet$ $\cal P$ is an {\it Arens property} if, regardless of the positive $m \geq 2$, the Banach lattices $E_1, \ldots, E_m$ and the form $A \in {\cal L}_r(E_1, \ldots, E_m)$, the Arens extension $A^{\ast[m+1]}$ of $A$ has $\cal P$ in the first variable, in the sense that the operator $A_{x_{2}^{\ast\ast}, \ldots,x_{m}^{\ast\ast}}\colon E_{1}^{\ast\ast}\longrightarrow \mathbb{R}$ has $\cal P$ for all $x_2^{\ast\ast} \in E_2^{**}, \ldots, x_m^{\ast\ast} \in E_m^{**}$.
\end{definition}

\begin{example}\rm Order continuity (Theorem \ref{quasesoc}(c)) and $\omega^*$-continuity \cite[p.\,413]{livrosean} are Arens properties.
\end{example}

\begin{theorem}\label{pro2} Let $\cal P$ be an Arens property, $m \geq 2$ and $E_{1}, \ldots,E_{m}$ be Banach lattices. Suppose that:\\
{\rm (i)} For $j=2,\ldots,m-1,$ and $ i=1,\ldots,m-j$, every regular linear operator from $E_{j}$ to $E_{j+i}^{\ast}$ is weakly compact;\\
{\rm (ii)} For all $k=2,\ldots,m$, $x_{1}^{\ast\ast}\in E_{1}^{\ast\ast}$ and   $T\in\mathcal{L}_{r}(E_{1};E_{k}^{\ast})$, the functional $T^{\ast\ast}(x_{1}^{\ast\ast})\in E_{k}^{\ast\ast\ast}$ has property $\mathcal{P}$.

 Then, for every form $A\in \mathcal{L}_{r}(E_{1},\ldots,E_{m})$, the Arens extension $A^{\ast[m+1]} \colon E_{1}^{\ast\ast}\times\cdots\times  E_{m}^{\ast\ast} \longrightarrow \mathbb{R}$  has $\mathcal{P}$-separately. 
\end{theorem}

\begin{proof} We shall proceed by induction on $m$. Given $A\in \mathcal{L}_{r}(E_{1},E_{2})$, $A^{***}$ has property $\mathcal{P}$ in the first variable because $\cal P$ is an Arens property. Let us prove that, for every  $x_{1}^{\ast\ast}\in E_{1}^{\ast\ast}$, $A^{***}(x_{1}^{\ast\ast},\bullet)\in E_{2}^{\ast}$ has property $\mathcal{P}$. Consider the regular linear operator  $T\colon E_{1}\longrightarrow E_{2}^{\ast}$, $T(x_{1})=A(x_{1},\bullet)$. For all $x_{2}^{\ast\ast}\in E_{2}^{\ast\ast}$ and $x_{1}\in E_{1}$,
$$T^{\ast}(x_{2}^{\ast\ast})(x_{1})=x_{2}^{\ast\ast}(T(x_{1}))=x_{2}^{\ast\ast}(A(x_{1},\bullet))=\overline{x_{2}^{\ast\ast}}^{\theta}(A)(x_{1}),$$
that is, $T^{\ast}(x_{2}^{\ast\ast})=\overline{x_{2}^{\ast\ast}}^{\theta}(A).$ So, for all $x_{1}^{\ast\ast}\in E_{1}^{\ast\ast}, x_{2}^{\ast\ast}\in E_{2}^{\ast\ast}$,
$$T^{\ast\ast}(x_{1}^{\ast\ast})(x_{2}^{\ast\ast})=x_{1}^{\ast\ast}(T^{\ast}(x_{2}^{\ast\ast}))=x_{1}^{\ast\ast}\big(\overline{x_{2}^{\ast\ast}}^{\theta}(A)\big)=\big(\overline{x_{1}^{\ast\ast}}^{\theta}\circ \overline{x_{2}^{\ast\ast}}^{\theta}\big)(A)=A^{***}(x_{1}^{\ast\ast},x_{2}^{\ast\ast}).$$
Since $T^{\ast\ast}(x_{1}^{\ast\ast})$ has property $\mathcal{P}$ by assumption, it follows that $A^{***}(x_{1}^{\ast\ast},\bullet)$ has property $\mathcal{P}$. This shows that the result holds for $m=2$.

Assume now that the result holds for $n$ and let us prove it holds for $n+1$.  
To do so we suppose that conditions (i) and (ii) hold for $n+1$. let $A\in \mathcal{L}_{r}(E_{1},\ldots,E_{n+1})$ be given. For every  $x_{i}^{\ast\ast}\in E_{i}^{\ast\ast}, i=2,\ldots,n+1$, we have
$$\overline{x_{i}^{\ast\ast}}^{\theta}\colon\mathcal{L}_{r}(E_{1},\ldots,E_{i})\longrightarrow \mathcal{L}_{r}(E_{1},\ldots,E_{i-1}),\, \overline{x_{i}^{\ast\ast}}^{\theta}(B)(x_{1},\ldots,x_{i-1})=x_{i}^{\ast\ast}(B(x_{1},\ldots,x_{i-1},\bullet)).$$
And for each $x_{1}^{\ast\ast}\in E_{1}^{\ast\ast}$, the functional  $\overline{x_{1}^{\ast\ast}}^{\theta}\colon E_{1}^{\ast}\longrightarrow \mathbb{R}$ is given by $\overline{x_{1}^{\ast\ast}}^{\theta}=x_{1}^{\ast\ast}$. Moreover,
\begin{align*}
A^{\ast[n+2]}(x_{1}^{\ast\ast},\ldots,x_{n+1}^{\ast\ast})&=\big(\overline{x_{1}^{\ast\ast}}^{\theta}\circ \cdots\circ \overline{x_{n+1}^{\ast\ast}}^{\theta}\big)(A)=\big(\overline{x_{1}^{\ast\ast}}^{\theta}\circ \cdots\circ \overline{x_{n}^{\ast\ast}}^{\theta}\big)\big(\overline{x_{n+1}^{\ast\ast}}^{\theta}(A)\big)\\
&=\big(\overline{x_{n+1}^{\ast\ast}}^{\theta}(A)\big)^{\ast[n+1]}(x_{1}^{\ast\ast},\ldots,x_{n}^{\ast\ast}).
\end{align*}
Since $\overline{x_{n+1}^{\ast\ast}}^{\theta}(A)\in\mathcal{L}_{r}(E_{1},\ldots,E_{n})$, by the induction hypothesis we have that  $\big(\overline{x_{n+1}^{\ast\ast}}^{\theta}(A)\big)^{\ast[n+1]}$ has  $\mathcal{P}$-separately, so  $A^{\ast[n+2]}$ has property $\mathcal{P}$ in the first $n$ variables. To prove that  $A^{\ast[n+2]}$ has property $\mathcal{P}$ in the $(n+1)$-th variable, let $x_{i}^{\ast\ast}\in E_{i}^{\ast\ast}, i=1,\ldots,n$, be given. Our job is to show that $A^{\ast[n+2]}(x_{1}^{\ast\ast},\ldots,x_{n}^{\ast\ast},\bullet)\colon E_{n+1}^{\ast\ast}\longrightarrow\mathbb{R}$ has property $\mathcal{P}$. Given $x_{i}\in E_{i}, i=1,\ldots,n-1$, consider the regular linear operator 
$$A_{x_{1},\ldots,x_{n-1}}\colon E_{n}\longrightarrow E_{n+1}^{\ast},\, A_{x_{1},\ldots,x_{n-1}}(x_{n})=A(x_{1},\ldots,x_{n},\bullet).$$
     Given $x_{n+1}^{\ast\ast}\in E_{n+1}^{\ast\ast}$, take a net $(x_{\alpha_n})_{\alpha_n}$ in $E_{n}$ such that  $J_{E_{n}}(x_{\alpha_n}) \stackrel{\omega^*} \longrightarrow x_{n}^{\ast\ast}$ and apply the  $\omega^{\ast}$-$\omega^{\ast}$-continuity of $[A_{x_{1},\ldots,x_{n-1}}]^{\ast\ast}$ and  
 Lemma \ref{lema} to obtain
\begin{align}
[A_{x_{1},\ldots,x_{n-1}}]^{\ast\ast}(x_{n}^{\ast\ast})(x_{n+1}^{\ast\ast})&
=\lim_{\alpha_{n}}[A_{x_{1},\ldots,x_{n-1}}]^{\ast\ast}(J_{E_{n}}(x_{\alpha_{n}}))(x_{n+1}^{\ast\ast})\nonumber\\
&=\lim_{\alpha_{n}}J_{E_{n+1}^{\ast}}(A_{x_{1},\ldots,x_{n-1}}(x_{\alpha_{n}}))(x_{n+1}^{\ast\ast})\nonumber\\
&=\lim_{\alpha_{n}}x_{n+1}^{\ast\ast}(A_{x_{1},\ldots,x_{n-1}}(x_{\alpha_{n}}))=\lim_{\alpha_{n}}x_{n+1}^{\ast\ast}(A(x_{1},\ldots,x_{n-1},x_{\alpha_{n}},\bullet))\nonumber\\
&=\lim_{\alpha_{n}}\overline{x_{n+1}^{\ast\ast}}^{\theta}(A)(x_{1},\ldots,x_{n-1},x_{\alpha_{n}})\nonumber\\
&=\lim_{\alpha_{n}} A^{\ast[n+2]}(J_{E_{1}}(x_{1}),\ldots,J_{E_{n-1}}(x_{n-1}),J_{E_{n}}(x_{\alpha_{n}}),x_{n+1}^{\ast\ast})\nonumber\\
&=A^{\ast[n+2]}(J_{E_{1}}(x_{1}),\ldots,J_{E_{n-1}}(x_{n-1}),x_{n}^{\ast\ast},x_{n+1}^{\ast\ast}). \label{e2}
\end{align}
For $x_{n}^{\ast\ast}\in E_{n}^{\ast\ast}$ and $x_{i}\in E_{i},i=1\ldots,n-2$, consider the regular linear operator $A_{x_{1},\ldots,x_{n-2},x_{n}^{\ast\ast}}\colon$ $ E_{n-1}\longrightarrow E_{n+1}^{\ast}$  given by
$$A_{x_{1},\ldots,x_{n-2},x_{n}^{\ast\ast}}(x_{n-1})(x_{n+1})=A^{\ast[n+2]}(J_{E_{1}}(x_{1}),\ldots,J_{E_{n-1}}(x_{n-1}),x_{n}^{\ast\ast},J_{E_{n+1}}(x_{n+1})).$$
On the one hand, for every $x_{n-1}\in E_{n-1}$ the functional $[A_{x_{1},\ldots,x_{n-2},x_{n}^{\ast\ast}}(x_{n-1})]^{\ast\ast}$ is a  $\omega^{\ast}$-continuous extension of  $A_{x_{1},\ldots,x_{n-2},x_{n}^{\ast\ast}}(x_{n-1})$. On the other hand, since  $A_{x_1, \ldots, x_{n-1}}$ is weakly compact by assumption,  
for every $x_{n}^{\ast\ast}\in E_{n}^{\ast\ast}$ the functional $[A_{x_{1},\ldots,x_{n-1}}]^{\ast\ast}(x_{n}^{\ast\ast})$ is $\omega^{\ast}$-continuous. Taking a net $(x_{\alpha_{n+1}})_{\alpha_{n+1}}$ in $E_{n+1}$ such that  $J_{E_{n+1}}(x_{\alpha_{n+1}}) \stackrel{\omega^*} \longrightarrow x_{n+1}^{\ast\ast}$, 
\begin{align}
[A_{x_{1},\ldots,x_{n-2},x_{n}^{\ast\ast}}(x_{n-1})]^{\ast\ast}(x_{n+1}^{\ast\ast})&=\lim_{\alpha_{n+1}}[A_{x_{1},\ldots,x_{n-2},x_{n}^{\ast\ast}}(x_{n-1})]^{\ast\ast}(J_{E_{n+1}}(x_{\alpha_{n+1}}))\nonumber\\
&=\lim_{\alpha_{n+1}}J_{E_{n+1}}(x_{\alpha_{n+1}})(A_{x_{1},\ldots,x_{n-2},x_{n}^{\ast\ast}}(x_{n-1}))\nonumber\\
&=\lim_{\alpha_{n+1}}A_{x_{1},\ldots,x_{n-2},x_{n}^{\ast\ast}}(x_{n-1})(x_{\alpha_{n+1}})\nonumber\\
&=\lim_{\alpha_{n+1}}A^{\ast[n+2]}(J_{E_{1}}(x_{1}),\ldots,J_{E_{n-1}}(x_{n-1}),x_{n}^{\ast\ast},J_{E_{n+1}}(x_{\alpha_{n+1}}))\nonumber\\
&=\lim_{\alpha_{n+1}}[A_{x_{1},\ldots,x_{n-1}}]^{\ast\ast}(x_{n}^{\ast\ast})(J_{E_{n+1}}(x_{\alpha_{n+1}}))\nonumber\\
&=[A_{x_{1},\ldots,x_{n-1}}]^{\ast\ast}(x_{n}^{\ast\ast})(x_{n+1}^{\ast\ast})\nonumber\\
&\stackrel{\rm  (\ref{e2})}{=}A^{\ast[n+2]}(J_{E_{1}}(x_{1}),\ldots,J_{E_{n-1}}(x_{n-1}),x_{n}^{\ast\ast},x_{n+1}^{\ast\ast}).\label{plkw}
\end{align}
Take a net $(x_{\alpha_{n-1}})_{\alpha_{n-1}}$ in $E_{n-1}$ such that  $J_{E_{n-1}}(x_{\alpha_{n-1}}) \stackrel{\omega^*} \longrightarrow x_{n-1}^{\ast\ast}$. Using that $[A_{x_{1},\ldots,x_{n-2},x_{n}^{\ast\ast}}]^{\ast\ast}$ is $\omega^{\ast}$-$\omega^{\ast}$-continuous and calling on Lemma \ref{lema}, for each  $x_{n+1}^{\ast\ast}\in E_{n+1}^{\ast\ast}$ we have
\begin{align}
[A_{x_{1},\ldots,x_{n-2},x_{n}^{\ast\ast}}]^{\ast\ast}(x_{n-1}^{\ast\ast})(x_{n+1}^{\ast\ast})&=\lim_{\alpha_{n-1}}[A_{x_{1},\ldots,x_{n-2},x_{n}^{\ast\ast}}]^{\ast\ast}(J_{E_{n-1}}(x_{\alpha_{n-1}}))(x_{n+1}^{\ast\ast})\nonumber\\
&=\lim_{\alpha_{n-1}}J_{E_{n+1}}(A_{x_{1},\ldots,x_{n-2},x_{n}^{\ast\ast}}(x_{\alpha_{n-1}}))(x_{n+1}^{\ast\ast})\nonumber\\
&=\lim_{\alpha_{n-1}}x_{n+1}^{\ast\ast}(A_{x_{1},\ldots,x_{n-2},x_{n}^{\ast\ast}}(x_{\alpha_{n-1}}))\nonumber\\
&=\lim_{\alpha_{n-1}}[A_{x_{1},\ldots,x_{n-2},x_{n}^{\ast\ast}}(x_{\alpha_{n-1}})]^{\ast\ast}(x_{n+1}^{\ast\ast})\nonumber\\
&\stackrel{\rm (\ref{plkw})}{=}\lim_{\alpha_{n-1}} A^{\ast[n+2]}(J_{E_{1}}(x_{1}),\ldots,J_{E_{n-1}}(x_{\alpha_{n-1}}),x_{n}^{\ast\ast},x_{n+1}^{\ast\ast})\nonumber\\
&=A^{\ast[n+2]}(J_{E_{1}}(x_{1}),\ldots,J_{E_{n-2}}(x_{n-2}),x_{n-1}^{\ast\ast},x_{n}^{\ast\ast},x_{n+1}^{\ast\ast}).\label{e3}
\end{align}
For $x_{n-1}^{\ast\ast}\in E_{n-1}^{\ast\ast}, x_{n}^{\ast\ast}\in E_{n}^{\ast\ast}$ and $x_{i}\in E_{i},i=1\ldots,n-3$, consider the regular linear operator $A_{x_{1},\ldots,x_{n-3},x_{n-1}^{\ast\ast},x_{n}^{\ast\ast}}\colon E_{n-2}\longrightarrow E_{n+1}^{\ast}$ given by
$$A_{x_{1},\ldots,x_{n-3},x_{n-1}^{\ast\ast},x_{n}^{\ast\ast}}(x_{n-2})(x_{n+1})=A^{\ast[n+2]}(J_{E_{1}}(x_{1}),\ldots,J_{E_{n-2}}(x_{n-2}),x_{n-1}^{\ast\ast},x_{n}^{\ast\ast},J_{E_{n+1}}(x_{n+1})).$$
On the one hand, for every $x_{n-2}\in E_{n-2}$ the functional $[A_{x_{1},\ldots,x_{n-3},x_{n-1}^{\ast\ast},x_{n}^{\ast\ast}}(x_{n-2})]^{\ast\ast}$ is a $\omega^{\ast}$-continuous extension of  $A_{x_{1},\ldots,x_{n-3},x_{n-1}^{\ast\ast},x_{n}^{\ast\ast}}(x_{n-2})$. On the other hand, since $A_{x_{1},\ldots,x_{n-2},x_{n}^{**}}$ is weakly compact by assumption, 
for every $x_{n-1}^{\ast\ast}\in E_{n}^{\ast\ast}$ the functional $[A_{x_{1},\ldots,x_{n-2},x_{n}]^{\ast\ast}}^{\ast\ast}(x_{n-1}^{\ast\ast})$ is $\omega^{\ast}$-continuous. So, 
\begin{align}
[A_{x_{1},\ldots,x_{n-3},x_{n-1}^{\ast\ast},x_{n}^{\ast\ast}}&(x_{n-2})]^{\ast\ast}(x_{n+1}^{\ast\ast})=\lim_{\alpha_{n+1}}[A_{x_{1},\ldots,x_{n-3},x_{n-1}^{\ast\ast},x_{n}^{\ast\ast}}(x_{n-2})]^{\ast\ast}(J_{E_{n+1}}(x_{\alpha_{n+1}}))\nonumber\\
&=\lim_{\alpha_{n+1}}J_{E_{n+1}}(x_{\alpha_{n+1}})(A_{x_{1},\ldots,x_{n-3},x_{n-1}^{\ast\ast},x_{n}^{\ast\ast}}(x_{n-2}))\nonumber\\
&=\lim_{\alpha_{n+1}}A_{x_{1},\ldots,x_{n-3},x_{n-1}^{\ast\ast},x_{n}^{\ast\ast}}(x_{n-2})(x_{\alpha_{n+1}})\nonumber\\
&=\lim_{\alpha_{n+1}}A^{\ast[n+2]}(J_{E_{1}}(x_{1}),\ldots,J_{E_{n-2}}(x_{n-2}),x_{n-1}^{\ast\ast},x_{n}^{\ast\ast},J_{E_{n+1}}(x_{\alpha_{n+1}}))\nonumber\\
&=\lim_{\alpha_{n+1}}[A_{x_{1},\ldots,x_{n-2},x_{n}^{\ast\ast}}]^{\ast\ast}(x_{n-1}^{\ast\ast})(J_{E_{n+1}}(x_{\alpha_{n+1}}))\nonumber\\
&=[A_{x_{1},\ldots,x_{n-2},x_{n}^{\ast\ast}}]^{\ast\ast}(x_{n-1}^{\ast\ast})(x_{n+1}^{\ast\ast})\nonumber\\
&\stackrel{\rm (\ref{e3})}{=}A^{\ast[n+2]}(J_{E_{1}}(x_{1}),\ldots,J_{E_{n-2}}(x_{n-2}),x_{n-1}^{\ast\ast},x_{n}^{\ast\ast},x_{n+1}^{\ast\ast})\label{uyew}.
\end{align}
     Since the operator $[A_{x_{1},\ldots,x_{n-3},x_{n-1}^{\ast\ast},x_{n}^{\ast\ast}}]^{\ast\ast}$ is $\omega^{\ast}$-$\omega^{\ast}$-continuous, for every $x_{n+1}^{\ast\ast}\in E_{n+1}^{\ast\ast}$, taking a net $(x_{\alpha_{n-2}})_{\alpha_{n-2}}$ in $E_{n-2}$ such that  $J_{E_{n-2}}(x_{\alpha_{n-2}}) \stackrel{\omega^*} \longrightarrow x_{n-2}^{\ast\ast}$, by Lemma \ref{lema} we have
\begin{align*}
[A_{x_{1},\ldots,x_{n-3},x_{n-1}^{\ast\ast},x_{n}^{\ast\ast}}]^{\ast\ast}(x_{n-2}^{\ast\ast})(x_{n+1}^{\ast\ast})&=\lim_{\alpha_{n-2}}[A_{x_{1},\ldots,x_{n-3},x_{n-1}^{\ast\ast},x_{n}^{\ast\ast}}]^{\ast\ast}(J_{E_{n-2}}(x_{\alpha_{n-2}}))(x_{n+1}^{\ast\ast})\\
&=\lim_{\alpha_{n-2}}J_{E_{n+1}}(A_{x_{1},\ldots,x_{n-3},x_{n-1}^{\ast\ast},x_{n}^{\ast\ast}}(x_{\alpha_{n-2}}))(x_{n+1}^{\ast\ast})\\
&=\lim_{\alpha_{n-2}}x_{n+1}^{\ast\ast}(A_{x_{1},\ldots,x_{n-3},x_{n-1}^{\ast\ast},x_{n}^{\ast\ast}}(x_{\alpha_{n-2}}))\\
&=\lim_{\alpha_{n-2}}[A_{x_{1},\ldots,x_{n-3},x_{n-1}^{\ast\ast},x_{n}^{\ast\ast}}(x_{\alpha_{n-2}})]^{\ast\ast}(x_{n+1}^{\ast\ast})\\
&\stackrel{\rm(\ref{uyew})}{=}\lim_{\alpha_{n-2}} A^{\ast[n+2]}(J_{E_{1}}(x_{1}),\ldots,J_{E_{n-2}}(x_{\alpha_{n-2}}),x_{n-1}^{\ast\ast},x_{n}^{\ast\ast},x_{n+1}^{\ast\ast})\\
&=A^{\ast[n+2]}(J_{E_{1}}(x_{1}),\ldots,J_{E_{n-3}}(x_{n-3}),x_{n-2}^{\ast\ast},x_{n-1}^{\ast\ast},x_{n}^{\ast\ast},x_{n+1}^{\ast\ast}).
\end{align*}
Repeating the procedure $(n-3)$ times, we end up with
\begin{equation}\label{e4}
[A_{x_{1},x_{3}^{\ast\ast},\ldots,x_{n}^{\ast\ast}}]^{\ast\ast}(x_{2}^{\ast\ast})(x_{n+1}^{\ast\ast})=A^{\ast[n+2]}(J_{E_{1}}(x_{1}),x_{2}^{\ast\ast},\ldots,x_{n+1}^{\ast\ast}),
\end{equation}
for every $x_{n+1}^{\ast\ast}\in E_{n+1}^{\ast\ast}$, where, for each  $x_{1}\in E_{1}$ and $x_{i}^{\ast\ast}\in E_{i}^{\ast\ast}, i=3,\ldots,n$,   $A_{x_{1},x_{3}^{\ast\ast},\ldots,x_{n}^{\ast\ast}}\colon$ $ E_{2}\longrightarrow E_{n+1}^{\ast}$ is the regular linear operator given by
$$A_{x_{1},x_{3}^{\ast\ast},\ldots,x_{n}^{\ast\ast}}(x_{2})(x_{n+1})= A^{\ast[n+2]}(J_{E_{1}}(x_{1}),J_{E_{2}}(x_{2}),x_{3}^{\ast\ast},\ldots,x_{n}^{\ast\ast},J_{E_{n+1}}(x_{n+1})).$$
Finally, given $x_{i}^{\ast\ast}\in E_{i}^{\ast\ast}, i=2,\ldots,n$, the regular linear operator $A_{x_{2}^{\ast\ast},\ldots,x_{n}^{\ast\ast}}\colon E_{1}\longrightarrow E_{n+1}^{\ast}$ defined by
$$A_{x_{2}^{\ast\ast},\ldots,x_{n}^{\ast\ast}}(x_{1})(x_{n+1})= A^{\ast[n+2]}(J_{E_{1}}(x_{1}),x_{2}^{\ast\ast},\ldots,x_{n}^{\ast\ast},J_{E_{n+1}}(x_{n+1})),$$
is weakly compact by condition (ii) for $n+1$. 
So, for every $x_{2}^{\ast\ast}\in E_{2}^{\ast\ast}$, $[A_{x_{1},x_{3}^{\ast\ast},\ldots,x_{n}^{\ast\ast}}]^{\ast\ast}(x_{2}^{\ast\ast})$ is $\omega^{\ast}$-continuous, therefore
\begin{align}
[A_{x_{2}^{\ast\ast},\ldots,x_{n}^{\ast\ast}}(x_{1})]^{\ast\ast}(x_{n+1}^{\ast\ast})&=\lim_{\alpha_{n+1}}[A_{x_{2}^{\ast\ast},\ldots,x_{n}^{\ast\ast}}(x_{1})]^{\ast\ast}(J_{E_{n+1}}(x_{\alpha_{n+1}}))\nonumber\\
&=\lim_{\alpha_{n+1}}J_{E_{n+1}}(x_{\alpha_{n+1}})(A_{x_{2}^{\ast\ast},\ldots,x_{n}^{\ast\ast}}(x_{1}))\nonumber\\
&=\lim_{\alpha_{n+1}}A_{x_{2}^{\ast\ast},\ldots,x_{n}^{\ast\ast}}(x_{1})(x_{\alpha_{n+1}})\nonumber\\
&=\lim_{\alpha_{n+1}} A^{\ast[n+2]}(J_{E_{1}}(x_{1}),x_{2}^{\ast\ast},\ldots,x_{n}^{\ast\ast},J_{E_{n+1}}(x_{\alpha_{n+1}}))\nonumber\\
&=\lim_{\alpha_{n+1}} [A_{x_{1},x_{3}^{\ast\ast},\ldots,x_{n}^{\ast\ast}}]^{\ast\ast}(x_{2}^{\ast\ast})(J_{E_{n+1}}(x_{\alpha_{n+1}}))\nonumber\\
&=[A_{x_{1},x_{3}^{\ast\ast},\ldots,x_{n}^{\ast\ast}}]^{\ast\ast}(x_{2}^{\ast\ast})(x_{n+1}^{\ast\ast})\nonumber\\
&\stackrel{\rm (\ref{e4})}{=}A^{\ast[n+2]}(J_{E_{1}}(x_{1}),x_{2}^{\ast\ast},\ldots,x_{n+1}^{\ast\ast}). \label{erny}
\end{align}
For the last time, taking a net $(x_{\alpha_{1}})_{\alpha_{1}}$ in $E_{1}$ such that  $J_{E_{1}}(x_{\alpha_{1}}) \stackrel{\omega^*} \longrightarrow x_{1}^{\ast\ast}$, the $\omega^{\ast}$-$\omega^{\ast}$ continuity of $[A_{x_{2}^{\ast},\ldots,x_{n}^{\ast\ast}}]^{\ast\ast}$ and  Lemma \ref{lema} give,  for every $x_{n+1}^{\ast\ast}\in E_{n+1}^{\ast\ast}$,
\begin{align*}
[A_{x_{2}^{\ast\ast},\ldots,x_{n}^{\ast\ast}}]^{**}(x_{1}^{\ast\ast})(x_{n+1}^{\ast\ast})&=\lim_{\alpha_{1}}[A_{x_{2}^{\ast\ast},\ldots,x_{n}^{\ast\ast}}]^{\ast\ast}(J_{E_{1}}(x_{\alpha_{1}})(x_{n+1}^{\ast\ast})\\
&=\lim_{\alpha_{1}}J_{E_{n+1}}(A_{x_{2}^{\ast\ast},\ldots,x_{n}^{\ast\ast}}(x_{\alpha_{1}}))(x_{n+1}^{\ast\ast})\\
&=\lim_{\alpha_{1}}x_{n+1}^{\ast\ast}(A_{x_{2}^{\ast\ast},\ldots,x_{n}^{\ast\ast}}(x_{\alpha_{1}}))\\
&=\lim_{\alpha_{1}}[A_{x_{2}^{\ast\ast},\ldots,x_{n}^{\ast\ast}}(x_{\alpha_{1}})]^{\ast\ast}(x_{n+1}^{\ast\ast})\\
&\stackrel{\rm (\ref{erny})}{=}\lim_{\alpha_{1}} A^{\ast[n+2]}(J_{E_{1}}(x_{\alpha_{1}}),x_{2}^{\ast\ast},\ldots,x_{n+1}^{\ast\ast})\\
&=A^{\ast[n+2]}(x_{1}^{\ast\ast},\ldots,x_{n+1}^{\ast\ast}).
\end{align*}
This proves that $[A_{x_{2}^{\ast\ast},\ldots,x_{n}^{\ast\ast}}]^{\ast\ast}(x_{1}^{\ast\ast}) = A^{\ast[n+2]}(x_{1}^{\ast\ast},\ldots,x_{n}^{\ast\ast},\bullet)$. 
By condition (ii) for $n+1$ we know that $[A_{x_{2}^{\ast\ast},\ldots,x_{n}^{\ast\ast}}]^{\ast\ast}(x_{1}^{\ast\ast})$ has property $\mathcal{P}$, so  $A^{\ast[n+2]}(x_{1}^{\ast\ast},\ldots,x_{n}^{\ast\ast},\bullet)$ has property $\mathcal{P}$, which completes the proof.
\end{proof}

Theorem \ref{pro2} gives sufficient conditions for Arens extensions of regular multilinear forms to be separately order continuous on the product of the whole of the biduals. Now we derive the case of vector-valued regular multilinear operators.
\begin{theorem}\label{ltheo} Let $m \geq 2$ and $E_{1}, \ldots,E_{m}$ be Banach lattices such that:\\
{\rm (i)} For $j=2,\ldots,m-1,$ and $ i=1,\ldots,m-j$, every regular linear operator from $E_{j}$ to $E_{j+i}^{\ast}$ is weakly compact;\\
{\rm (ii)} For all $k=2,\ldots,m$, $x_{1}^{\ast\ast}\in E_{1}^{\ast\ast}$ and   $T\in\mathcal{L}_{r}(E_{1};E_{k}^{\ast})$, the functional $T^{\ast\ast}(x_{1}^{\ast\ast})$ is order continuous on $E_{k}^{\ast\ast\ast}$.

 Then, for every Banach lattice $F$ and any $A\in \mathcal{L}_{r}(E_{1},\ldots,E_{m};F)$, the Arens extension $A^{\ast[m+1]}$ is separately order continuous on $E_{1}^{\ast\ast}\times\cdots\times  E_{m}^{\ast\ast}$. 
\end{theorem}

\begin{proof} Let  $A\in \mathcal{L}_{r}(E_{1},\ldots,E_{m};F)$ and $y^{\ast}\in F^{\ast}$ be given. Since $y^{\ast}\circ A\in \mathcal{L}_{r}(E_{1},\ldots,E_{m})$ and order continuity is an Arens property, by Theorem \ref{pro2} the extension  $(y^* \circ A)^{\ast[m+1]}$ is separately order continuous. For all $x_{i}^{\ast\ast}\in E_{i}^{\ast\ast}, i=1,\ldots,m$,
$$A^{\ast[m+1]}(x_{1}^{\ast\ast},\ldots,x_{m}^{\ast\ast})(y^{\ast})=(y^{\ast}\circ A)^{*[m+1]}(x_{1}^{\ast\ast},\ldots,x_{m}^{\ast\ast}).$$
For each $j\in\{1,\ldots,m\}$ let $x_{j}^{\ast\ast}\in E_{j}^{\ast\ast}$ and let $(x_{\alpha_{j}}^{\ast\ast})_{\alpha_{j}\in \Omega_{j}}$ be a net $E_{j}^{\ast\ast}$ such that $x_{\alpha_{j}}^{\ast\ast} \xrightarrow{\,\, o \,\, } 0$. There exists a net $(y_{\alpha_{j}}^{\ast\ast})_{\alpha_{j}\in \Omega_{j}}$ in $E_{j}^{\ast\ast}$ and $\alpha_{j_{0}}\in \Omega_{j}$ so that $y_{\alpha_{j}}^{\ast\ast}\downarrow 0$ and $|x_{\alpha_{j}}^{\ast\ast}|\leq y_{\alpha_{j}}^{\ast\ast}$ for every $\alpha_{j}\geq \alpha_{j_{0}}$. Without loss of generality, assume that $A$  and $y^{\ast}$ are positive. Since $(y^{\ast}\circ A)^{*[m+1]}(|x_{1}^{\ast\ast}|,\ldots,\bullet,\ldots,|x_{m}^{\ast\ast}|)\colon E_{j}^{\ast\ast}\longrightarrow \mathbb{R}$ is positive and order continuous,
$$A^{*[m+1]}(|x_{1}^{\ast\ast}|,\ldots,y_{\alpha_{j}}^{\ast\ast},\ldots,|x_{m}^{\ast\ast}|)(y^{\ast})
=(y^{\ast}\circ A)^{*[m+1]}(|x_{1}^{\ast\ast}|,\ldots,\bullet,\ldots,|x_{m}^{\ast\ast}|)(y_{\alpha_{j}}^{\ast\ast})\downarrow 0.$$
It follows that  $A^{*[m+1]}(|x_{1}^{\ast\ast}|,\ldots,y_{\alpha_{j}}^{\ast\ast},\ldots,|x_{m}^{\ast\ast}|) \downarrow 0$ \cite[Theorem 1.18]{positiveoperators} and, for every $\alpha_{j}\geq \alpha_{j_{0}}$,
\begin{align*}
|A^{*[m+1]}(x_{1}^{\ast\ast},\ldots,x_{\alpha_{j}}^{\ast\ast},\ldots,x_{m}^{\ast\ast})|
&\leq A^{*[m+1]}(|x_{1}^{\ast\ast}|,\ldots,|x_{\alpha_{j}}^{\ast\ast}|,\ldots,|x_{m}^{\ast\ast}|)\\
&\leq A^{*[m+1]}(|x_{1}^{\ast\ast}|,\ldots,y_{\alpha_{j}}^{\ast\ast},\ldots,|x_{m}^{\ast\ast}|) \downarrow 0.
\end{align*}
This shows that $A^{*[m+1]}(x_{1}^{\ast\ast},\ldots,x_{\alpha_{j}}^{\ast\ast},\ldots,x_{m}^{\ast\ast})
\xrightarrow{\,\, o \,\, } 0$ and proves that $A^{*[m+1]}$ is separately order continuous.
\end{proof}

\begin{example}\label{exxe}\rm As to condition (i) above, we have the following examples between nonreflexive Banach lattices:\\
(a) Every operator from $c_0$ to $c_0^*=\ell_1$ is compact, hence weakly compact (this is Pitt's Theorem). \\
(b) Every operator from $C(K)$, where $K$ is a compact Hausdorff space, to a KB-space is weakly compact. Just recall that KB-spaces do not contain a copy of $c_0$ \cite[Theorem 4.60]{positiveoperators} and apply \cite[Theorem 5]{pelczynski}.\\
(c) Since any AM-space with order unity is order isometric to a $C(K)$-space \cite[Theorem 4.29]{positiveoperators}, from (b) it follows that every operator from an AM-space with order unity to a KB-space is weakly compact. And since the dual of an AM-space is a KB-space, every operator from an AM-space with order unity to its dual is weakly compact. In particular, every operator from $\ell_\infty$ to $\ell_\infty^*$ is weakly compact.
\end{example}

\begin{corollary}
 Let $m\geq 2$ and  $E_{1}, \ldots,E_{m}, F$ be Banach lattices such that  every regular operator from $E_{j}$ to $E_{j+i}^{\ast}$ is weakly compact, $j=2,\ldots,m-1, i=1,\ldots,m-j$. If  $E_{1}^{\ast}$ has order continuous norm, then the Arens extension  $A^{\ast[m+1]}$ of any operator $A\in \mathcal{L}_{r}(E_{1},\ldots,E_{m};F)$ is separately order continuous on $E_{1}^{\ast\ast}\times\cdots\times  E_{m}^{\ast\ast}$.
\end{corollary}

\begin{proof} Condition (i) of Theorem \ref{ltheo} is given by assumption. For $x_{1}^{\ast\ast}\in E_{1}^{\ast\ast}$ and   $T\in\mathcal{L}_{r}(E_{1};E_{k}^{\ast})$, $x_{1}^{\ast\ast}$ is order continuous because the norm of $E_{1}^{\ast}$ is order continuous \cite[Theorem 2.4.2]{nieberg}. Since $T^*$ is order continuous \cite[Theorem 1.73]{positiveoperators},  $T^{**}(x_1^{**}) =  x_1^{**} \circ T^* $ is order continuous as well, so condition (ii) is fulfilled too.
\end{proof}

Recall that a Banach space $E$ is {\it Arens regular} if every bounded linear operator from $E$ to $E^*$ is weakly compact (see, e.g., \cite{livrosean}). The Banach lattices $c_0, \ell_\infty$ and $C(K)$, where $K$ is a compact Hausdorff space, in particular AM-spaces with order unit, are Arens regular (cf. Example \ref{exxe}).

\begin{corollary} Let $E$ be an Arens regular Banach lattice. Then, for every Banach lattice $F$, the Arens extension $A^{\ast[m+1]}$ of any regular $m$-linear operator $A \colon E^m \longrightarrow F$ is separately order continuous on $(E^{**})^m$.
\end{corollary}

\begin{proof} The Arens regularity of $E$ gives condition (i) of Theorem \ref{ltheo} right away and implies that, for every $T\in\mathcal{L}_{r}(E;E^{\ast})$, $T^{**}(E^{**}) \subseteq J_{E^*}(E^*) \subseteq (E^{**})^*_n$, which gives condition (ii).
\end{proof}

We finish the paper with one more result on order continuity of Arens extensions of homogeneous polynomials. For a polynomial $P \in {\cal P}_r(^mE;F)$, we write $P^{*[m+1]} := AR_m^{\theta}(P)$. Recall that $P$ is {\it orthogonally additive} if $P(x+y) = P(x) + P(y)$ whenever $x$ and $y$ are disjoint. The literature on orthogonally additive polynomials is vast.

A linear operator $u\colon E \longrightarrow E^*$ is {\it symmetric} if $u(x)(y) = u(y)(x)$ for all $x,y \in E$. A Banach space $E$ is {\it symetrically Arens regular} if every symmetric operator from $E$ to $E^*$ is weakly compact. Of course, Arens regular spaces are symmetrically Arens regular, but there are symmetrically Arens regular spaces that fail to be Arens regular \cite{leung}.

\begin{proposition} Let $E,F$ be Banach lattices and $P \in {\cal P}_r(^mE;F)$. If either $P$ is orthogonally additive and $F = \mathbb{R}$ or $E$ is symmetrically Arens regular, then the Arens extension $P^{\ast[m+1]} \colon E^{**} \longrightarrow F^{**}$ of $P$ is order continuous on $E^{**}$.
\end{proposition}

\begin{proof} Assume first that $P$ is orthogonally additive and $F = \mathbb{R}$. By Proposition \ref{respol} we know that $P^{\ast[m+1]}$ is order continuous at the origin on $E^{**}$,  therefore it is order continuous at every point of $E^{**}$ by \cite[Proposition 8]{nakano}.

Suppose now that $E$ is symmetrically Arens regular. It is plain that we can assume that $P$ is positive. We know that $(\check{P})^{\ast[m+1]}$ is order continuous in the first variable on $E^{**}$ (Theorem \ref{quasesoc}) and positive because $\check{P}$ is positive. In order to check that it is symmetric, let $\rho \in S_m$ be given.  
For every $\varphi \in F^*$, since $E$ is symmetrically Arens regular and $\varphi \circ \check{P}$ is symmetric, by \cite[Theorem 8.3]{acg} (or \cite[Corollary 6]{csgv}) we know that $ AR_m^{\theta}(\varphi \circ \check{P})$ is symmetric as well. So, for $x_1^{**}, \ldots, x_m^{**} \in E^{**}$,
\begin{align*} (\check{P})^{\ast[m+1]}(x_1^{**}, \ldots, x_m^{**} )&(\varphi)  = AR_m^{\theta}(\check{P})(x_1^{**}, \ldots, x_m^{**})(\varphi)  = AR_m^{\theta}(\varphi \circ \check{P})(x_1^{**}, \ldots, x_m^{**})\\
& = AR_m^{\theta}(\varphi \circ \check{P})(x_{\rho(1)}^{**}, \ldots, x_{\rho(m)}^{**})\\
& =  AR_m^{\theta}(\check{P})(x_{\rho(1)}^{**}, \ldots, x_{\rho(m)}^{**})(\varphi) = (\check{P})^{\ast[m+1]}(x_{\rho(1)}^{**}, \ldots, x_{\rho(m)}^{**} )(\varphi),
\end{align*}
proving that $(\check{P})^{\ast[m+1]}$ is symmetric. The order continuity in the first variable and the symmetry yield that the positive $m$-linear operator $(\check{P})^{\ast[m+1]}$ is separately order continuous on $(E^{**})^m$. By \cite[Lemma 2.6]{Buskes}  it follows that $(\check{P})^{\ast[m+1]}$ is jointly order continuous on $(E^{**})^m$. So, if $x_{\alpha}^{**}\xrightarrow{\,\, o \,\,} x^{**}$ in $E^{**}$, then
\begin{align*}P^{\ast[m+1]}(x_{\alpha}^{**}) & = (\check{P})^{\ast[m+1]}(x_{\alpha}^{**}, \ldots, x_{\alpha}^{**}) \xrightarrow{\,\, o \,\,} (\check{P})^{\ast[m+1]}(x^{**}, \ldots, x^{**}) = P^{\ast[m+1]}(x^{**}).
\end{align*}
\end{proof}

\medskip

\noindent {\bf Acknowledgement.} The authors are grateful to R. Ryan for pointing out a gap in the first version of the paper and for drawing our attention to \cite{nakano}.

\bigskip

\noindent Faculdade de Matem\'atica~~~~~~~~~~~~~~~~~~~~~~Instituto de Matem\'atica e Estat\'istica\\
Universidade Federal de Uberl\^andia~~~~~~~~ Universidade de S\~ao Paulo\\
38.400-902 -- Uberl\^andia -- Brazil~~~~~~~~~~~~ 05.508-090 -- S\~ao Paulo -- Brazil\\
e-mail: botelho@ufu.br ~~~~~~~~~~~~~~~~~~~~~~~~~e-mail: luisgarcia@ime.usp.br

%
%
%


\begin{thebibliography}{99}\small

\vspace*{-0.5em}

\bibitem{abramovich} { Y. Abramovich and G. Sirotkin}, {\it On order convergence of nets}, Positivity \textbf{9} (2005), 287-292.

\vspace*{-0.5em}

\bibitem{li} C. D. Aliprantis and O. Burkinshaw, {\it Locally Solid Riesz Spaces with Applications to Economics}, Math Surveys and Monographs, Volume 105, American Math. Society, 2003.

\vspace*{-0.5em}

\bibitem{positiveoperators} C. D. Aliprantis and D. O. Burkinshaw, {\it Positive Operators}, Springer, Dordrecht, 2006.

\vspace*{-0.5em}

\bibitem{arens}  R. Arens, {\it The adjoint of a bilinear operation}, Proc. Amer. Math. Soc. {\bf 2} (1951), 839--848.



\vspace*{-0.5em}

\bibitem{aronberner} R. Aron and P. Berner, {\it A Hahn-Banach extension theorem for analytic mappings}, Bull. Soc. Math. France  {\bf 106}  (1978), no. 1, 3--24.

\vspace*{-0.5em}

\bibitem{acg} R. Aron, B. Cole and T. Gamelin, {\it Spectra of algebras of analytic functions on a Banach space}, J. Reine Angew. Math. {\bf 415} (1991), 51--93.

\vspace*{-0.5em}

\bibitem{lg} G. Botelho and L. A. Garcia. {\it Bidual extensions of Riesz multimorphisms}, preprint (available at arXiv:2108.03769v2[math.FA], 2021).

\vspace*{-0.5em}

\bibitem{ryan1} C. Boyd, R. A. Ryan and N. Snigireva, {\it Synnatzschke’s theorem for polynomials}, Positivity  \textbf{25}, 229–242, 2021.

\vspace*{-0.5em}

\bibitem{nakano} C. Boyd, R. A. Ryan and N. Snigireva, {\it A Nakano Carrier theorem for polynomials}, preprint (available at arXiv:2107.10337v1[math.FA], 2021).

\vspace*{-0.5em}

\bibitem{bu} Q. Bu and G. Buskes, {\it Polynomials on Banach lattices and positive tensor products}, J. Math. Anal.
Appl. \textbf{388} (2012), 845–862.

\vspace*{-0.5em}

\bibitem{Buskes} G. Buskes and S. Roberts, {\it Arens extensions for polynomials
and the Woodbury–Schep formula}, Positivity and Noncommutative Analysis, 37–48, Trends Math., Birkhäuser, 2019.

\vspace*{-0.5em}

\bibitem{csgv} F. Cabello S\'anchez, R. Garc\'ia and I. Villanueva, {\it Extension of multilinear operators on Banach spaces}, Extracta Math. {\bf 15} (2000), no. 2, 291--334.

\vspace*{-0.5em}

\bibitem{gamelin} A. M. Davie, T. W. Gamelin, {\it A theorem on polynomial-star approximation}, Proc. Amer. Math. Soc. {\bf 106} (1989), no. 2, 351--356.

\vspace*{-0.5em}

\bibitem{livrosean} S. Dineen, {\it Complex Analysis on Infinite Dimensional Spaces}, Springer, 1999.

\vspace*{-0.5em}

\bibitem{kus} A. G. Kusraev and Z. A. Kusraeva, {\it Factorization of order bounded
disjointness preserving multilinear
operators}, Modern methods in operator theory and harmonic analysis, 217–236, Springer Proc. Math. Stat., 291, Springer, 2019.

\vspace*{-0.5em}

\bibitem{leung} D. H. Leung, {\it Some remarks on regular Banach spaces}, Glasgow Math. J. {\bf 38} (1996), 243--248.

\vspace*{-0.5em}

\bibitem{loane} J. Loane, {\it Polynomials on vector lattices}, PhD Thesis, National University of Ireland, Galway, 2007.




\vspace*{-0.5em}

\bibitem{nieberg}  P. Meyer--Nieberg. {\it Banach Lattices}, Springer-Verlag, 1991.

\vspace*{-0.5em}

\bibitem{pelczynski} A. Pe{\l}czy\'nski, {\it Projections in certain Banach spaces}, Studia Math. {\bf 19} (1960), 209--228.

\vspace*{-0.5em}

\bibitem{schaefer} H. H. Schaefer, {\it Banach Lattices and Positive Operators}, Springer, 1974.

\vspace*{-0.5em}

\bibitem{yilma} R. Yilmaz. {\it The Arens triadjoints of some bilinear maps}, Filomat {\bf 28} (2014), 963-979.
%
%
%


	

%
%
%
%







%
%


%






%
%
%
%
%
%
%
%
%
%
%
%


%
%
%























%




%
%
%
%
%

\end{thebibliography}
\end{document}